\documentclass[letterpaper, 11pt]{amsart}
\usepackage[margin=1.2in]{geometry}
\usepackage{amssymb,latexsym,eufrak,amsmath,amscd, graphics}
\usepackage{xspace,xcolor}
\usepackage[breaklinks,colorlinks,citecolor=teal,linkcolor=teal,urlcolor=teal,pagebackref,hyperindex]{hyperref}
\usepackage[alphabetic]{amsrefs} 
\usepackage{mathrsfs}
\usepackage{amsfonts}
\usepackage{tikz-cd} 
\usepackage{mathtools} 
\usepackage{stmaryrd} 
\usetikzlibrary{calc, intersections} 

\tikzset{symbol/.style={draw=none,every to/.append style={edge node={node [sloped, allow upside down, auto=false]{$#1$}}}}} 

\newcommand{\into}{\hookrightarrow}
\newcommand{\onto}{\twoheadrightarrow}
\newcommand{\from}{\leftarrow}

\tikzset{
    labl/.style={anchor=north, rotate=90, inner sep=1mm}
}

\let\proof\noindentitproof 

\def\Mustata{Mus\-ta\-\c{t}\u{a}\xspace}

\DeclareMathOperator{\bbc}{\mathbb{C}}

\DeclareMathOperator{\bbz}{\mathbb{Z}}
\DeclareMathOperator{\bbq}{\mathbb{Q}}

\DeclareMathOperator{\bba}{\mathbb{A}}

\DeclareMathOperator{\os}{\mathcal{O}}
\DeclareMathOperator{\vs}{\mathcal{V}}
\DeclareMathOperator{\as}{\mathcal{A}}
\DeclareMathOperator{\bs}{\mathcal{B}}
\DeclareMathOperator{\ls}{\mathcal{L}}
\DeclareMathOperator{\cs}{\mathcal{C}}
\DeclareMathOperator{\es}{\mathcal{E}}
\DeclareMathOperator{\ds}{\mathcal{D}}
\DeclareMathOperator{\hs}{\mathcal{H}}
\DeclareMathOperator{\rs}{\mathcal{R}}
\DeclareMathOperator{\ms}{\mathcal{M}}
\DeclareMathOperator{\ns}{\mathcal{N}}
\DeclareMathOperator{\fs}{\mathcal{F}}
\DeclareMathOperator{\gs}{\mathcal{G}}
\DeclareMathOperator{\qs}{\mathcal{Q}}

\DeclareMathOperator{\is}{\mathcal{I}}
\DeclareMathOperator{\ks}{\mathcal{K}}

\DeclareMathOperator{\bff}{\bf{f}}

\DeclareMathOperator{\bft}{\bf{t}}
\DeclareMathOperator{\bfh}{\bf{h}}

\DeclareMathOperator{\sEnd}{\mathcal{E}{\it nd}}
\DeclareMathOperator{\sDer}{\mathcal{D}{\it er}}

\DeclareMathOperator{\BS}{\operatorname{BS}}

\DeclareMathOperator{\lct}{\operatorname{lct}}

\DeclareMathOperator{\res}{\operatorname{res}}

\DeclareMathOperator{\FT}{\operatorname{FT}}

\DeclareMathOperator{\Ann}{\operatorname{Ann}}

\DeclareMathOperator{\ord}{\operatorname{ord}}

\DeclareMathOperator{\lcm}{\operatorname{lcm}}

\DeclareMathOperator{\im}{\operatorname{im}}

\DeclareMathOperator{\coker}{\operatorname{coker}}

\DeclareMathOperator{\gr}{\operatorname{gr}}

\DeclareMathOperator{\GL}{\operatorname{GL}}

\DeclareMathOperator{\Spec}{\operatorname{Spec}}

\DeclareMathOperator{\Der}{\operatorname{Der}}

\DeclareMathOperator{\rSpec}{\mathcal{S}\mkern-1mu\it{pec}}

\newtheorem{lemma}{Lemma}[section]
\newtheorem{theorem}[lemma]{Theorem}
\newtheorem{corollary}[lemma]{Corollary}

\theoremstyle{definition}
\newtheorem{definition}[lemma]{Definition}

\newtheorem{remark}[lemma]{Remark}

\begin{document}

\author[Jonghyun Lee]{Jonghyun Lee}
\address{Department of Mathematics, University of Michigan,
Ann Arbor, MI 48109, USA}
\email{{nuyhgnoj@umich.edu}}

\vspace*{-3em}

\begin{abstract}
We obtain estimates for the roots of the Bernstein-Sato polynomial in terms of a log resolution by an orbifold.
\end{abstract}

\title[Estimates for roots of Bernstein-Sato polynomials]{Estimates for roots of Bernstein-Sato polynomials via log resolutions by orbifolds}

\maketitle

\vspace{-2em}

\section{Introduction}

Given a nonzero regular function $f\in \os_X(X)$ on an irreducible smooth complex variety $X$, the Bernstein-Sato polynomial of $f$ is the monic polynomial $b_f(s)$ of minimal degree satisfying
$$
b_f(s)  f^s \in \ds_X[s] \cdot f^{s+1},
$$
where $\ds_X$ is the sheaf of differential operators on $X$, which acts on the formal symbol $f^s$ in the expected way. The roots of $b_f(s)$ are negative rational numbers \cite{Kas76}. If $f$ is not invertible, then by specializing $s$ to $-1$, it is easy to see that $b_f(-1)=0$. The negative of the largest root of $b_f(s)/(s+1)$ is the \textit{minimal exponent} $\widetilde{\alpha}(f)$ (with the convention this is $\infty$ if $b_f(s)=s+1$). Introduced independently by Bernstein \cite{Ber72} and Sato, the Bernstein-Sato polynomial is an important invariant of singularities. For example, the log canonical threshold of $f$ is given by $\lct(f) = \min(1,\widetilde{\alpha}(f))$ \cite[Theorem 10.6]{Kol97}, and $\widetilde{\alpha}(f)>1$ if and only if the hypersurface defined by $f$ has rational singularities \cite{Sai93}. Moreover, the minimal exponent characterizes the higher rational singularities of the hypersurface defined by $f$ (see \cite{JKSY22} and \cite{MOPW23}) as well as the higher Du Bois singularities (see \cite[Appendix]{FL24} and \cite{MP25}).

By refining a method introduced by Kashiwara in \cite{Kas76}, Lichtin \cite{Lic89} obtained estimates for the roots of $b_f(s)$ in terms of a strong log resolution of the pair $(X,D)$, where $D$ is the divisor defined by $f$, as follows. Given a projective morphism $\pi: Y \to X$ that is an isomorphism over the complement of $D$ such that $Y$ is smooth and $\pi^*D = \sum_{i\in I} a_i E_i$ is a simple normal crossings divisor, we can write $K_{Y/X}=\sum_{i\in I} k_i E_i$, where $K_{Y/X}$ is the relative canonical divisor. Then every root of $b_f(s)$ is of the form $- \frac{k_i+1+\ell}{a_i}$ for some $i$ and $\ell \in \bbz_{\geq 0}$. Moreover, if $D$ is reduced and its strict transform on $Y$ is smooth, then by using the theory of Hodge ideals, \Mustata and Popa \cite{MP20} showed that $\widetilde{\alpha}(f)\geq \min_{i\in I, k_i>0} \frac{k_i+1}{a_i}$ (see also \cite{DM22} for a proof that builds on Kashiwara and Lichtin's method).  

In this paper, we extend Kashiwara and Lichtin's method to the case when $Y$ is an orbifold. Our approach is also robust enough to obtain similar estimates for the Bernstein-Sato polynomial of an arbitrary subvariety that was introduced in \cite{BMS06}. Before we state our main result, we give one of its implications for hypersurfaces. 

Let $X$ be an irreducible smooth affine variety with global coordinates $x_1,\dots,x_n\in \os_X(X)$. Given a nonzero regular function $f \in \os_X(X)$, let $\pi\colon Y \to X$ be a projective morphism that is an isomorphism over the complement of the divisor defined by $f$ such that $Y$ is irreducible and has quotient singularities. Let $U/G \onto Y$ be an \'{e}tale cover, where $G$ is a finite group acting on a smooth affine variety $U$ (such an \'{e}tale cover always exists; see Appendix \ref{quotient singularities}). Suppose there exists a finite open cover $U=\bigcup_{i\in I} U_i$, so that if we denote the morphism $U_i \to X$ by $\pi_i$, then each $U_i$ has global coordinates $y_{1},\dots,y_{n}\in \os_U(U_i)$ and invertible functions $u,v\in \os_U(U_i)$ such that $\pi_i^*f = u y_{1}^{a_{i1}} \cdots y_{n}^{a_{in}}$ and the Jacobian of $\pi_i$ is given by $v y_{1}^{k_{i1}}\cdots y_{n}^{k_in}$.

\begin{theorem}\label{main result for hypersurfaces}
With the above notation, every root of $b_f(s)$ is of the form $- \frac{k_{ij}+1+\ell}{a_{ij}}$ for some $i\in I$, $j\in \{1,\dots,n\}$, and $\ell\in \bbz_{\geq 0}$. Moreover, if $a_{i1}\in \{0,1\}$ and $k_{i1}=0$ for every $i\in I$, then
$$
\widetilde{\alpha}(f) \geq \min_{i\in I,j\geq 2,a_{ij}\neq 0}\frac{k_{ij}+1}{a_{ij}}.
$$
\end{theorem}

The above theorem recovers \cite[Theorem 5]{Lic89} and \cite[Corollary D]{MP20} by taking $G$ to be the trivial group.

We now turn to the main result of this paper. As before, let $X$ be an irreducible smooth affine variety with global coordinates $x_1,\dots,x_n\in \os_X(X)$. Given nonzero regular functions $f_1,\dots,f_r\in \os_X(X)$, we denote by $Z$ the closed subscheme of $X$ defined by the ideal generated by $f_1,\dots,f_r$. Let $\pi\colon Y\to X$ be a projective morphism that is an isomorphism over the complement of $Z$ such that $Y$ is irreducible and has quotient singularities with \'{e}tale cover $U/G \onto Y$, where $G$ is a finite group acting on a smooth affine variety $U$. Given a finite open cover $U=\bigcup_{i\in I} U_i$ such that each $U_i$ has global coordinates $y_1,\dots,y_n\in \os_U(U_i)$, we denote by $\pi_i$ the morphism $U_i \to X$ and by $g_i=\det ( (\frac{\partial (\pi_i^* x_p)}{\partial y_q})_{pq} )\in \os_U(U_i)$ the Jacobian of $\pi_i$. We consider the $\ds_X\langle t_1,\dots,t_r,\partial_{t_1},\dots,\partial_{t_r}\rangle$-module
$$
B_{\bff} = \bigoplus_{\beta \in \bbz_{\geq 0}^r} \os_X \partial_t^{\beta}\delta_{\bff},
$$
where $\bff$ denotes the $r$-tuple $(f_1,\dots,f_r)$ (for the precise definition of $B_{\bff}$, see Section \ref{bf}). Similarly, we consider the $\ds_{U_i}\langle t_1,\dots,t_r,\partial_{t_1},\dots,\partial_{t_r}\rangle$-module $B_{\pi_i^*\bff}=\bigoplus_{\beta \in \bbz_{\geq 0}^r} \os_{U_i} \partial_t^{\beta}\delta_{\pi_i^*\bff}$ for each $i\in I$, where $\pi_i^*\bff$ denotes the $r$-tuple $(\pi_i^*f_1,\dots,\pi_i^*f_r)$.

Given a section $u= \sum_{\beta} u_{\beta}\partial_t^{\beta}\delta_{\bff} \in \Gamma(X,B_{\bff})$, we consider its $b$-function $b_u(s)\in \bbc[s]$ (for the definition, see Section \ref{b-functions}). For example, if $u=\delta_{\bff}$, then $b_{\delta_{\bff}}(s)$ is the Bernstein-Sato polynomial of the subvariety $Z$ \cite{BMS06}. In particular, if $r=1$, then $b_{\delta_{\bff}}(s)$ is the Bernstein-Sato polynomial $b_{f_1}(s)$ of $f_1$. For each $i\in I$, we put
$$
\pi_i^*u := \sum_{\beta}  \pi_i^*(u_{\beta}) g_i \partial_t^{\beta}\delta_{\pi_i^*\bff} \in \Gamma(U_i,B_{\pi_i^*\bff})
$$
and consider its $b$-function $b_{\pi^*_iu}(s)\in \bbc[s]$. The main result of this paper relates the $b$-function of $u$ with those of each $\pi_i^*u$.

\begin{theorem}\label{main result}
With the above notation, if we set $b(s) = \lcm_{i\in I} b_{\pi_i^*u}(s)$, then
$$
b_u(s) \ | \ b(s)b(s+1)\cdots b(s+N)
$$
for some integer $N \geq 0$.
\end{theorem}

The argument used to prove Theorem \ref{main result} extends, with only minor modifications, to Bernstein-Sato ideals. We first recall the definition of Bernstein-Sato ideals and then state the corresponding theorem. 

Let $g,f_1,\dots, f_r\in \os_X(X)$ be nonzero regular functions on a smooth variety $X$ and let $\bff$ denote the $r$-tuple $(f_{1}, \dots,f_{r})$. The Bernstein-Sato ideal $\BS_{\bff,g}$ of $\bff$ relative to $g$ is the ideal in $\bbc[s_{1}, \dots,s_{r}]$ consisting of all polynomials $b(s_1,\dots, s_{r})$ such that 
$$
b(s_1,\dots,s_{r}) gf_1^{s_1}\cdots f_r^{s_{r}} \in \ds_X[s_1,\dots,s_{r}] \cdot gf_1^{s_1+1}\cdots f_r^{s_{r}+1}.
$$
When $g=1$, $\BS_{\bff,1}$ is the usual Bernstein-Sato ideal of $\bff$, which has been extensively studied in, for example, \cite[]{May97}, \cite[]{BM99}, \cite[]{BVWZ21a}, and \cite[]{BVWZ21b}. Sabbah \cite{SabI} showed that $\BS_{\bff,1}$ contains a nonzero product of certain linear forms (see also \cite{Gyo93}). Also, Budur, van der Veer, and Van Werde \cite{BVVW24} showed that the codimension one irreducible components of the zero locus of $\BS_{\bff,1}$ in $\bbc^{r}$ are affine hyperplanes that can be expressed in terms of a strong log resolution of $(X,D)$, where $D$ is the divisor defined by $\prod_{i=1}^rf_i$.

Before stating our result for Bernstein-Sato ideals, we introduce the following setup. Let $X$ be an irreducible smooth affine variety with global coordinates $x_{1},\dots,x_{n}\in\os_{X}(X)$. Given nonzero regular functions $g,f_1,\dots, f_r\in \os_X(X)$, let $D$ be the divisor defined by $\prod_{i=1}^rf_i$ and let $\pi\colon Y \to X$ be a projective morphism that is an isomorphism over the complement of $D$ such that $Y$ is irreducible and has quotient singularities. Let $p\colon U/G \onto Y$ be an \'{e}tale cover, where $G$ is a finite group acting on a smooth affine variety $U$, and let $q$ denote the morphism $U\to U/G$. Suppose that $(\pi\circ p\circ q)^{*}D+K_{U/X}$ is a simple normal crossings divisor, where $K_{U/X}$ is the relative canonical divisor, locally defined by the Jacobian of the morphism $\pi\circ p\circ q\colon U\to X$. For each prime divisor $E\subset U$, let $k_{E}\in \bbz_{\geq 0}$ denote the coefficient of $E$ in $K_{U/X}$, and, for a regular function $h\in \os_{X}(X)$, let $\ord_{E}(h)\in \bbz_{\geq 0}$ denote the order of vanishing of $(\pi\circ p\circ q)^{*}h$ along $E$. Finally, we denote the $r$-tuple $(f_{1},\dots,f_{r})$ by $\bff$.

\begin{theorem}\label{main result for bernstein-sato ideals}
With the above notation, $\BS_{\bff, g}$ contains a nonzero polynomial of the form
$$
\prod_{E}\prod_{\ell=0}^{N}\prod_{\ell'=0}^{N'}\prod_{m=1}^{\sum_{i=1}^{r}\ord_{E}(f_{i})}(\ord_{E}(f_{1})s_{1}+\cdots + \ord_{E}(f_{r})s_{r}+ \ord_{E}(g)+k_{E}+m+\ell+\ell')
$$
for some integers $N,N'\geq 0$, where the outermost product runs over the irreducible components $E$ of $(\pi\circ p\circ q)^{*}D$. Moreover, if $g=1$, then we may take $N'=0$.
\end{theorem}

When $g=1$ and the group $G$ is trivial, the above theorem  implies a special case of \cite[Theorem 1.2]{BVVW24}.

We now explain how our approach differs from that of Kashiwara \cite{Kas76} and Lichtin \cite{Lic89}. Their approach relies crucially on the theory of holonomic $\ds$-modules. We expect that the arguments in \cite{Kas76} and \cite{Lic89}, together with the auxiliary results established in Sections \ref{direct summands of canonical modules} and \ref{etale pullback}, can be used to prove Theorem \ref{main result for hypersurfaces}. However, in the setting of Theorem \ref{main result} when $r >1$, the previous approaches break down because the relevant $\ds_{X}$-modules that were holonomic in the case $r=1$, are no longer so. The main idea of our approach is to replace holonomicity with finite generation over certain sheaves of rings associated to the $V$-filtration.

For the proof of Theorem \ref{main result}, it is important to consider Bernstein-Sato polynomials and $\ds$-modules on singular varieties. There has been extensive work on Bernstein-Sato polynomials on singular varieties; see, for example, \cite{HM18}, \cite{AHN17},\cite{BJNB19}, \cite{AHJNTW22}, and \cite{Dir25}. In particular, we use the theory of differential direct summands, as introduced and studied in \cite{AHN17}, to relate $b$-functions on $U/G$ with those on the smooth variety $U$. While the sheaf $\ds_Z$ of differential operators on a singular variety $Z$ typically lacks finiteness properties, it turns out that the associated graded of $\ds_Y$ with respect to its order filtration is finitely generated (see Lemma \ref{quot sing finite type}), which allows us to relate $b$-functions on $X$ with those on the orbifold $Y$. 

As mentioned above, Theorem \ref{main result for hypersurfaces} is obtained as a consequence of Theorem \ref{main result} in this paper. In a companion paper \cite{Lee26}, we obtain an analogue of Theorem \ref{main result for hypersurfaces} for the minimal exponent of a local complete intersection, as introduced and studied in \cite{CDMO24}, which we use to compute the minimal exponents of semi-quasihomogeneous and Khovanskii non-degenerate complete intersections.

\subsection{Outline of the paper}

Section \ref{background} covers various background material on quasicoherent bimodules and algebras, rings of differential operators, the $V$-filtration, the canonical sheaf, the $\ds$-module $B_{\bff}$, and the theory of differential direct summands. In Section \ref{bfomega} we define the right $\ds$-module analogue $B^{\omega}_{\bff}$ of $B_{\bff}$ on normal varieties. In Section \ref{direct summands of canonical modules} we show that the natural pullback map on canonical modules associated to a finite group quotient of a smooth affine variety is a differential direct summand, which allows us to relate $b$-functions on the quotient with those on the smooth variety. In Section \ref{etale pullback} we show that $b$-functions of sections of $B_{\bff}^{\omega}$ are invariant under pullback by surjective \'{e}tale morphisms. In Section \ref{coherent modules and good filtrations} we discuss formal results on coherent modules and good filtrations over certain sheaves of rings associated to the $V$-filtration. In Section \ref{transfer bimodule and direct image} we introduce transfer bimodules and direct images adapted for the sheaves of rings discussed in Section \ref{coherent modules and good filtrations}, and in Section \ref{projective direct image} we show that projective direct image preserves coherence over these sheaves of rings. In Section \ref{proof of the main theorem} we prove Theorems \ref{main result for hypersurfaces} and \ref{main result}. Finally, in Section \ref{Estimates of zero loci of Bernstein-Sato ideals} we carry over the necessary definitions and results from Bernstein-Sato polynomials to Bernstein-Sato ideals and then prove Theorem \ref{main result for bernstein-sato ideals}.

\subsection{Conventions}

All varieties in this paper are reduced (and not necessarily irreducible) schemes of finite type over $\bbc$. Given $\beta\in \bbz^r$, we denote its components by $\beta=(\beta_1,\dots,\beta_r)$ and put $|\beta|=\sum_{i=1}^r \beta_i$ and $\beta! = \prod_{i=1}^r \beta_i!$.

\subsection{Acknowledgements}

I am grateful to Mircea \Mustata for many valuable discussions. I would also like to thank Gary Hu for helpful feedback on an earlier version of the introduction.


\section{Background}\label{background}

In this section, we review the necessary background and definitions for this paper.

\subsection{Quasicoherent bimodules and sheaves of rings}

For a sheaf of rings $\rs$ on a topological space $X$, we say that an $(\rs,\rs)$-bimodule is an $\rs$-bimodule.

\begin{definition}
We say that an $\os_X$-bimodule $\ms$ on a scheme $X$ is quasicoherent if it is quasicoherent both as a left and a right $\os_X$-module.
\end{definition}

\begin{definition}
We say that a sheaf $\as$ of (possibly noncommutative) rings on a scheme $X$ is quasicoherent if it is equipped with a ring morphism $\os_{X}\to\as$ such that $\as$ is quasicoherent as an $\os_X$-bimodule, where the bimodule structure is given by left and right multiplication by $\os_X$.
\end{definition}

Given a quasicoherent sheaf of rings $\as$ on a scheme $X$ and a (left or right) $\as$-module $\ms$, we say that $\ms$ is quasicoherent if it is so after restricting scalars along $\os_X \to \as$.

\subsection{Rings of differential operators}

We use Grothendieck's definition of rings of differential operators \cite[Section 16.8]{Gro67}. Given a morphism of schemes $f: X \to Y$, let $\ds_{X/Y}$ denote the sheaf of $f^{-1}\os_Y$-linear differential operators on $X$, and denote its order filtration by $(F_p\ds_{X/Y})_{p\in \bbz}$. Given a ring map $k \to R$, let $D_{R/k}$ denote the ring of $k$-linear differential operators of $R$, and denote its order filtration by $(F_pD_{R/k})_{p\in \bbz}$. For a variety $X$, we write $\ds_X := \ds_{X/\Spec \bbc}$, and for a $\bbc$-algebra $R$, we write $D_R =D_{\Spec R} := D_{R/\bbc}$. For general facts about $\ds_X$-modules when $X$ is a smooth variety, we refer to \cite{HTT08}.

For a variety $X$, we say that regular functions $x_1,\dots,x_n\in \os_X(X)$ are global coordinates of $X$ if the morphism $\phi: X \to \bba^n=\Spec \bbc[t_1,\dots,t_n]$ given by $\phi^*t_i=x_i$ is \'{e}tale, in which case we denote the corresponding partial derivatives by $\partial_{x_1},\dots,\partial_{x_n} \in \sDer_{\bbc}(\os_X)$ and put $\partial_x^{\alpha} = \prod_{i=1}^n \partial_{x_i}^{\alpha_i}$ for $\alpha\in \bbz_{\geq 0}^n$, so that $\ds_X= \bigoplus_{\alpha\in \bbz^n_{\geq 0}} \os_X\partial_x^{\alpha}$. 

We now consider $\bba^r=\Spec \bbc[t_1,\dots,t_r]$ and its ring of differential operators
$$
D_{\bba^r} = \bbc\langle t_1,\dots,t_r, \partial_{t_1},\dots,\partial_{t_r}\rangle = \bigoplus_{\alpha,\beta\in \bbz_{\geq 0}^r} \bbc t^{\alpha}\partial_t^{\beta},
$$
where $t^{\alpha}=\prod_{i=1}^r t_i^{\alpha_i}$ and $\partial_t^{\beta}=\prod_{i=1}^r \partial_{t_i}^{\beta_i}$. Given any $\bbc$-algebra $R$, we have the following natural isomorphism of $\bbc$-algebras by \cite[Proposition II.58]{QG21a}:
\begin{equation}\label{local product}
D_{R[t_1,\dots,t_r]} \simeq D_R \otimes_{\bbc} D_{\bba^r}=\bigoplus_{\alpha,\beta\in \bbz_{\geq 0}^r} D_R t^{\alpha}\partial_t^{\beta}.
\end{equation}
Thus, given any variety $X$, if we denote the projection $X\times \bba^r \to X$ by $p_X$, then 
$$
p_{X,*}\ds_{X\times \bba^r}= \ds_X \otimes_{\bbc} D_{\bba^r} = \bigoplus_{\alpha,\beta\in \bbz_{\geq 0}^r} \ds_X t^{\alpha}\partial_t^{\beta}.
$$
Note that if $\ns$ is a $\ds_{X\times \bba^r}$-module, then $p_{X,*}\ns$ is a $p_{X,*}\ds_{X\times \bba^r}$-module.

\begin{remark}\label{qcoh diff}
If $f: X \to Y$ is a morphism of schemes that is locally of finite presentation, then $\ds_{X/Y}$ is a quasicoherent sheaf of rings on $X$ by \cite[Proposition 16.8.6]{Gro67}.
\end{remark}

\begin{remark}\label{right action formula}
Given a variety $X$ and a right $\ds_X$-module $\ns$, observe that 
$$
(vf) \cdot \theta = v\cdot (f\cdot \theta) = v \cdot ([f,\theta] + \theta\cdot f) = - v \theta(f) + (v \cdot \theta) f
$$
for $v\in \ns$, $f\in \os_X$, and $\theta \in \sDer_{\bbc}(\os_X)$.
\end{remark}

\subsection{The $V$-filtration}\label{the v-filtration}

Throughout this section, we fix a variety $X$ and an integer $r\geq 1$. In this section, we recall the definition and some basic properties of the $V$-filtration on $\ds_{X\times \bba^r}$.

We denote the standard coordinates on $\bba^r$ by $t_1,\dots,t_r$ and the projection $X \times \bba^r \to X$ by $p_X$, and put
$$
\rs_X := p_{X,*}\ds_{X\times \bba^r} = \ds_X\otimes_{\bbc} D_{\bba^r}.
$$
We also denote by $\is \subset \os_{X\times \bba^r}$ the coherent sheaf of ideals generated by $t_1,\dots,t_r$.

The $V$-filtration on $\ds_{X\times \bba^r}$ is the decreasing filtration $(V^k\ds_{X\times \bba^r})_{k\in \bbz}$ given by
$$
V^k\ds_{X\times \bba^r} := \{P \in \ds_{X\times \bba^r} \ | \ P ( \is^j) \subset \is^{j+k} \text{ for all $j\in \bbz$}\},
$$
where $\is^j=\os_{X\times \bba^r}$ for $j\leq 0$. We also consider the $V$-filtration $(V^k\rs_X)_{k\in \bbz}$ on $\rs_X$ given by
$$
V^k\rs_X :=p_{X,*} V^k\ds_{X\times \bba^r}.
$$
We define the $\bbz_{\geq 0}$-graded Rees algebra of $\rs_X$ as
$$
R^+_V\rs_X := \bigoplus_{k\geq 0} V^k\rs_X T^k \subset \rs_X[T].
$$

\begin{lemma}\label{qcoh V filtration}
$V^k\ds_{X \times \bba^r}$ is a quasicoherent $\os_{X \times \bba^r}$-bimodule for each $k$, hence $V^0\ds_{X\times \bba^r}$, $V^0\rs_X$, and $R^+_V\rs_X$ are quasicoherent sheaves of rings.
\end{lemma}
\proof The following proof is based on the observations made in \cite[Section 1]{Sab87}. Let $Y=\rSpec_{X \times \bba^r} (\bigoplus_{j\in \bbz} \is^jT^j)$ and consider the natural morphism $\pi: Y \to X\times\bba^{r}$. Also consider the morphism $Y \to \bba^1=\Spec \bbc[t]$, where the coordinate $t$ pulls back to $T^{-1}$. Note that $V^k\ds_{X\times \bba^r}$ identifies with the subsheaf $(\pi_*\ds_{Y/\bba^1})_k\subset \pi_*\ds_{Y/\bba^1}$ of homogeneous differential operators of degree $k$. Since 
$\pi_*\ds_{Y/\bba^1} = \bigoplus_{k\in \bbz} (\pi_*\ds_{Y/\bba^1})_k$, it follows from Remark \ref{qcoh diff} and Lemma \ref{qcoh direct summands} that $V^k\ds_{X\times \bba^r} \simeq (\pi_*\ds_{Y/\bba^1})_k$ is a quasicoherent $\os_{X \times \bba^r}$-bimodule. \qed

We omit the proof of the following well-known lemma.

\begin{lemma}\label{explicit V filtration}
For $k\in \bbz$, we have
$$
V^k\rs_X =\bigoplus_{|\alpha|-|\beta|\geq k} \ds_X t^{\alpha}\partial_t^{\beta}=\bigoplus_{|\alpha|-|\beta|\geq k} \ds_X \partial_t^{\beta} t^{\alpha},
$$
where both direct sums run over all $\alpha,\beta\in \bbz^r_{\geq 0}$ such that $|\alpha|-|\beta|\geq k$.
\end{lemma}

For $k\geq 0$, it is immediate from Lemma \ref{explicit V filtration} that
\begin{equation}\label{push generation from zero}
V^k\rs_X = (t_1,\dots,t_r)^k V^0\rs_X = \sum_{|\alpha|=k} t^{\alpha}V^0\rs_X.
\end{equation}

\begin{lemma}\label{generation from zero}
We have $V^k\ds_{X\times \bba^r} = \is^k\cdot V^0\ds_{X\times \bba^r}$ for $k\geq 0$.
\end{lemma}
\proof Since both sides of the equation are left $\os_{X\times \bba^r}$-quasicoherent submodules of $\ds_{X\times \bba^r}$, it suffices to check that $p_{X,*}(V^k\ds_{X\times \bba^r}) = p_{X,*}(\is^k\cdot V^0\ds_{X\times \bba^r})$ in $p_{X,*}(\ds_{X\times \bba^r})$, but this follows from \eqref{push generation from zero}. \qed

We put
$$
V^kD_{\bba^r} := \bigoplus_{|\alpha|-|\beta|\geq k} \bbc t^{\alpha} \partial_t^{\beta} \subset D_{\bba^r}
\ \ \ \ 
\text{and}
\ \ \ \ 
R^+_VD_{\bba^r}:=\bigoplus_{k\geq 0}V^kD_{\bba^r}T^k \subset D_{\bba^r}[T].
$$
By Lemma \ref{explicit V filtration}, we have $V^k\rs_X = \ds_X\otimes_{\bbc} V^kD_{\bba^r}$ and $R^+_V\rs_X= \ds_X \otimes_{\bbc} R^+_VD_{\bba^r}$.

\subsection{The canonical sheaf}\label{dmod on normal}

Given a smooth variety $X$ of pure dimension $n$, the canonical sheaf $\omega_X=\wedge^n \Omega^1_X$ has a natural right $\ds_X$-module structure extending its $\os_X$-module structure, given as follows. If $U \subset X$ is an open subset with global coordinates $x_1,\dots,x_n\in \os_X(U)$, then for every $g\in \os_X(U)$ we have
\begin{equation}\label{canonical dmod}
gdx_1\wedge \cdots \wedge dx_n \cdot \partial_{x_i} = - \partial_{x_i}(g) dx_1\wedge \cdots \wedge dx_n.
\end{equation}

Now let $X$ be a normal variety of pure dimension $n$. We denote the smooth locus and singular locus of $X$ by $X_{sm}$ and $X_{sing}$, respectively, and the inclusion $X_{sm} \to X$ by $j$. The canonical sheaf of $X$ is defined as $\omega_X:=j_*\omega_{X_{sm}}$. Note that $j_*\omega_{X_{sm}}$ is a right $j_*\ds_{X_{sm}}$-module. As in \cite[Theorem 3.2]{QG21b}, we define the right $\ds_X$-module structure on $\omega_X$ by restricting scalars along $\ds_X \to j_*j^*\ds_X=j_*\ds_{X_{sm}}$. By \cite[Proposition II.2.2]{Lev81} (see also the proof of \cite[Theorem 3.2]{QG21b}), this morphism is in fact an isomorphism, which we record as follows:
\begin{equation}\label{pull push normal isom}
\ds_X \xrightarrow{\sim} j_*\ds_{X_{sm}}.
\end{equation}

Given a morphism $f: X \to Y$ between normal varieties of the same pure dimension such that $f^{-1}(Y_{sing})$ has codimension $\geq 2$, we define a $\bbc$-linear map 
\begin{equation}\label{pullback forms}
f^*: \Gamma(Y,\omega_Y) \to \Gamma(X,\omega_X)
\end{equation}
as follows. Let $U = X_{sm} \cap f^{-1}(Y_{sm})$ and denote the morphism $f|_U: U \to Y_{sm}$ by $f'$. Consider the morphism $f'^*\omega_{Y_{sm}} \to \omega_{U}$ obtained by taking the determinant of the natural morphism $f'^*\Omega^1_{Y_{sm}} \to \Omega^1_{U}$. Taking global sections of the corresponding morphism $\omega_{Y_{sm}} \to f'_*\omega_{U}$ yields a map $\Gamma(Y_{sm},\omega_{Y_{sm}}) \to \Gamma(U,\omega_U)$, which we denote by $f'^*$. Note that $\Gamma(X,\omega_X)=\Gamma(U,\omega_U)$ since $X\setminus U = X_{sing} \cup f^{-1}(Y_{sing})$ has codimension $\geq 2$. We then define the map in \eqref{pullback forms} as the composition
$$
\Gamma(Y,\omega_Y)= \Gamma(Y_{sm},\omega_{Y_{sm}}) \xrightarrow{f'^*} \Gamma(U,\omega_U)=\Gamma(X,\omega_X).
$$

It is easy to see that if $f: X \to Y$ and $g: Y \to Z$ are morphisms between normal varieties of the same pure dimension such that $f^{-1}(Y_{sing})$, $g^{-1}(Z_{sing})$, and $(gf)^{-1}(Z_{sing})$ all have codimension $\geq 2$, then for every section $v\in \Gamma(Z,\omega_Z)$ we have 
\begin{equation}\label{composition}
f^*(g^*v) = (g\circ f)^*v \in \Gamma(X,\omega_X).
\end{equation}

\subsection{The left $\ds$-module $B_{\bff}$}\label{bf}

In this section we recall the definition of the left $\ds$-module $B_{\bff}$ discussed in the Introduction (for details, we refer to \cite[Section 2]{CDMO24}).

Let $X$ be a smooth variety. Given regular functions $f_1,\dots,f_r\in \os_X(X)$, we denote the $r$-tuple $(f_1,\dots,f_r)$ by $\bff$ (when $r=1$, we simply denote the $1$-tuple $(f_1)$ by $f_1$), and consider the graph embedding
$$
i_{\bff}: X \to X \times \bba^r, x \mapsto (x,\bff(x))
$$
and the $\ds$-module theoretic pushforward $i_{\bff,+}\os_X$, which is a left $\ds_{X\times \bba^r}$-module. Note that $i_{\bff,+}\os_X = \hs^r_{\Gamma_{\bff}}(\os_{X\times \bba^r})$, where $\Gamma_{\bff}=i_{\bff}(X) \subset X \times \bba^r$. We denote the standard coordinates on $\bba^r$ by $t_1,\dots,t_r$ and the projection $X \times \bba^r \to X$ by $p_X$, and put $\rs_X = p_{X,*}\ds_{X\times \bba^r}$. We define the left $\rs_X$-module $B_{\bff}$ as
$$
B_{\bff} := p_{X,*}(i_{\bff,+}\os_X) = \bigoplus_{\beta\in \bbz_{\geq 0}^r} \os_X \partial_t^{\beta}\delta_{\bff}.
$$
Here, the actions of $\os_X$ and $\partial_{t_i}$ are the obvious ones, while the actions of $\theta \in \sDer_{\bbc}(\os_X)$ and $t_i$ are given by 
\begin{equation}\label{bf action}
\theta \cdot g\partial_t^{\beta}\delta_{\bff} = \theta(g)\partial_t^{\beta}\delta_{\bff} - \sum_{i=1}^r \theta(f_i) g \partial_t^{\beta+e_i}\delta_{\bff}
\ \ \
\text{and}
\ \ \
t_i \cdot g\partial_t^{\beta}\delta_{\bff} = f_i g \partial_t^{\beta}\delta_{\bff} - \beta_i g \partial_t^{\beta-e_i}\delta_{\bff},
\end{equation}
where $e_1,\dots,e_r$ is the standard basis of $\bbz^r$. By \cite[(6)]{CDMO24}, we have a natural isomorphism of $\rs_{X}$-modules
\begin{equation}\label{specialization}
B_{\bff}[1/f_{1}\cdots f_{r}]\simeq \os_{X}[1/f_{1}\cdots f_{r},s_{1},\dots,s_{r}]f_{1}^{s_{1}}\cdots f_{r}^{s_{r}}
\end{equation}
that maps $\delta_{\bff}$ to $f_{1}^{s_{1}}\cdots f_{r}^{s_{r}}$, where $\rs_{X}$ acts on the right-hand side in the expected way.

\subsection{$B$-functions}\label{b-functions}

Throughout this section, we fix a variety $X$ and an integer $r\geq 1$. In this section we recall the definition and some basic properties of $b$-functions. We denote the standard coordinates on $\bba^r$ by $t_1,\dots,t_r$ and consider the operator $s = - \sum_{i=1}^r \partial_{t_i} t_i \in V^0D_{\bba^r}$. We put $\rs_X = \ds_X\otimes_{\bbc} D_{\bba^r}$.

Given a left $\rs_X$-module $\ms$ and a section $u\in \Gamma(X,\ms)$, the $b$-function $b_u(s)\in \bbc[s]$ of $u$ is defined as the monic generator of the ideal
$$
\{b(s) \in \bbc[s] \ | \ b(s) \cdot u \in V^1\rs_X \cdot u\},
$$
with the convention that $0\in \bbc[s]$ is monic. Similarly, given a right $\rs_X$-module $\ns$ and a section $v\in \Gamma(X,\ns)$, the $b$-function $b_v(s)\in \bbc[s]$ of $v$ is defined as the monic generator of the ideal $\{b(s) \in \bbc[s] \ | \ v\cdot b(s) \in v\cdot V^1\rs_X\}$. 

\begin{remark}\label{b function local}
The $b$-function is local in the following sense. Given an open cover $X=\bigcup_i U_i$, a right $\rs_X$-module $\ns$, and a section $v\in \Gamma(X,\ns)$, we have $b_v(s) = \lcm_i b_{v|_{U_i}}(s)$, that is, $b_v(s)$ is the monic generator of the ideal $\bigcap_i (b_{v|_{U_i}}(s))$. 
\end{remark}

\begin{remark}\label{b function affine local}
The $b$-function is affine-local in the following sense. Given a quasicoherent right $\rs_X$-module $\ns$ and a section $v\in \Gamma(X,\ns)$, if $X$ is affine, then $b_v(s)$ is the monic generator of the ideal $\{b(s) \in \bbc[s] \ | \ v \cdot b(s) \in v\cdot (D_X \otimes_{\bbc}V^1D_{\bba^r})\}$.
\end{remark}

\begin{lemma}\label{right moving}
Given a right $\rs_X$-module $\ns$ and a section $v\in \Gamma(X,\ns)$, we have
$$
(v\cdot V^k\rs_X) b_v(s-k) \subset v\cdot V^{k+1}\rs_X
$$
for every $k\in \bbz$. 
\end{lemma}
\proof For $\alpha,\beta\in \bbz_{\geq 0}^r$, we have $t^{\alpha}\partial_t^{\beta} b_v(s-|\alpha|+|\beta|)=b_v(s)t^{\alpha}\partial_t^{\beta}$ by \cite[Lemmas 7.3 and 7.4]{CDMO24}. Since $V^k\rs_X = V^{k+1}\rs_X + \sum_{|\alpha|-|\beta|=k} \ds_X t^{\alpha}\partial_t^{\beta}$ by Lemma \ref{explicit V filtration}, we have
$$
V^k\rs_X b_v(s-k) \subset b_v(s) V^k\rs_X + V^{k+1}\rs_X.
$$
Thus
$$
(v\cdot V^k\rs_X)b_v(s-k) \subset v \cdot b_v(s) V^k\rs_X + v \cdot V^{k+1}\rs_X = v \cdot V^{k+1}\rs_X. \qed
$$

\begin{remark}\label{b function annihilator}
Given a right $\rs_X$-module $\ns$ and a section $v \in \Gamma(X,\ns)$, by Lemma \ref{right moving}, we see that $b_v(s)$ is the monic generator of the ideal $\Ann_{\bbc[s]}(v\cdot V^0\rs_X/v\cdot V^1\rs_X)$. Similarly, given a left $\rs_X$-module $\ms$ and a section $u \in \Gamma(X,\ms)$, $b_u(s)$ is the monic generator of the ideal $\Ann_{\bbc[s]}(V^0\rs_X\cdot u/V^1\rs_X \cdot u)$.
\end{remark}

\begin{remark}
Given a quasicoherent left $\ds_{X\times \bba^r}$-module $\ms$ and a section $u\in \Gamma(X\times \bba^r,\ms)$, if we denote the projection $X \times \bba^r \to X$ by $p_X$ and view $u$ as a global section of $p_{X,*}\ms$, then we have
$$
\Ann_{\bbc[s]}(V^0\ds_{X\times \bba^r}\cdot u/V^1\ds_{X\times \bba^r}\cdot u) = \Ann_{\bbc[s]}(V^0\rs_X \cdot u / V^1\rs_X \cdot u).
$$
Indeed, if we put $\qs=V^0\ds_{X\times \bba^r}\cdot u/V^1\ds_{X\times \bba^r}\cdot u$, then $\qs$ is scheme-theoretically supported on $X \times 0 = V(t_1,\dots,t_r) \subset X \times \bba^r$ and $p_{X,*}\qs
= V^0\rs_X \cdot u / V^1\rs_X \cdot u$.
Thus the $b$-function introduced in \cite{BMS06} agrees with ours.
\end{remark}

\begin{remark}\label{rational roots}
Let $f_1,\dots,f_r\in \os_X(X)$ be regular functions on a smooth variety $X$ and let $\bff$ denote the $r$-tuple $(f_1,\dots,f_r)$. It follows from the results in \cite{Kas83} (see also \cite{BMS06}) that for every $u\in \Gamma(X,B_{\bff})$, the $b$-function $b_u(s)$ is a nonzero polynomial such that all its roots are rational.
\end{remark}


\subsection{Differential direct summands}

In this section, we recall the theory of differential direct summands developed in \cite{AHN17}, adapted to the setting of right $D$-modules.

We say that $R \subset S$ is a direct summand of $\bbc$-algebras with splitting $\rho$ if $R\subset S$ is an inclusion of commutative $\bbc$-algebras that is split as $R$-modules by $\rho: S \to R$. Given a direct summand $R \subset S$ of $\bbc$-algebras with splitting $\rho$ and an element $Q\in D_S$, note that
$$
\rho \circ Q|_R: R \xrightarrow{Q|_R} S \xrightarrow{\rho} R
$$
is an element of $D_R$ (see \cite[Lemma 3.1]{AHN17}). 

\begin{definition}[Differential direct summand {\cite[Definition 3.2]{AHN17}}]
Let $R \subset S$ be a direct summand of $\bbc$-algebras with splitting $\rho$. A differential direct summand compatible with $\rho$ is a pair $(M\subset N, \tau)$, where $M$ is a right $D_R$-module, $N$ is a right $D_S$-module, $M$ is an $R$-submodule of $N$, and $\tau: N \to M$ is an $R$-module map, called a differential splitting, such that
$$
\tau(v \cdot Q) = v\cdot (\rho \circ Q|_R)
$$
for all $Q \in D_S$ and $v \in M$.
\end{definition}

For a differential direct summand $(M\subset N, \tau)$ compatible with $\rho$, we sometimes suppress the differential splitting $\tau$ and simply write $M\subset N$.

\begin{definition}[Morphism of differential direct summands {\cite[Definition 3.5]{AHN17}}]
Let $R \subset S$ be a direct summand of $\bbc$-algebras with splitting $\rho$. A morphism 
$$
(M\subset N,\tau) \to (M'\subset N', \tau')
$$
of differential direct summands compatible with $\rho$ consists of a $D_S$-module map $\phi: N \to N'$ such that $\phi(M) \subset M'$, $\phi|_M: M \to M'$ is $D_R$-linear, and the following diagram commutes:
\[
\begin{tikzcd}
M \arrow[r,hook] \arrow[d,"\phi|_M"] & N \arrow[r,"\tau"] \arrow[d,"\phi"] & M \arrow[d,"\phi|_M"] \\
M' \arrow[r,hook]& N'  \arrow[r,"\tau'"] & M'.
\end{tikzcd}
\]
\end{definition}

\begin{lemma}\label{Proposition 3.6}
Let $R\subset S$ be a direct summand of $\bbc$-algebras with splitting $\rho$. Given a differential direct summand $(M\subset N,\tau)$ compatible with $\rho$ and an element $f\in R$, the localization map $N\to N_{f}$ induces a morphism $(M\subset N,\tau) \to(M_{f}\subset N_{f},\tau_{f})$ of differential direct summands compatible with $\rho$.
\end{lemma}
\proof After noting that the $D_{R}$-action on $M_f$ is defined inductively on the order by the rule
$$
\frac{v}{f^t} \cdot \delta = \frac{v \cdot \delta + \tfrac{v}{f^t}\cdot [\delta,f^t]}{f^t}
$$
for $\delta \in D_{R}$, $v\in M$, and $t\geq 0$, the proof is identical to that of \cite[Proposition 3.6]{AHN17}, so we omit it. \qed

\begin{lemma}\label{Lemma 3.7}
Let $R\subset S$ be a direct summand of $\bbc$-algebras with splitting $\rho$. Given a morphism of differential direct summands $(M\subset N,\tau) \to (M'\subset N', \tau')$ given by $\phi: N \to N'$, we have that $\ker(\phi|_{M})\subset\ker(\phi)$ and $\coker(\phi|_{M})\subset\coker(\phi)$ are differential direct summands compatible with $\rho$. Furthermore, the inclusion and projection maps are morphisms of differential direct summands.
\end{lemma}
\proof The proof is identical to that of \cite[Lemma 3.7]{AHN17}, so we omit it. \qed

By Lemma \ref{Lemma 3.7}, the category of differential direct summands compatible with $\rho$ is abelian. In the proof of the following lemma, given a chain complex 
$$
\cdots  \to (M^{i-1}\subset N^{i-1},\tau^{i-1}) \to (M^i\subset N^i,\tau^i) \to (M^{i+1}\subset N^{i+1},\tau^{i+1}) \to \cdots  
$$
of differential direct summands compatible with $\rho$, we suppress the differential splittings $\tau^i$ and simply write $M^{\bullet} \subset N^{\bullet}$.

\begin{lemma}\label{local coh direct summand}
Given a direct summand $R \subset S$ of noetherian $\bbc$-algebras with splitting $\rho$ and an ideal $I\subset R$, if $(M \subset N,\tau)$ is a differential direct summand compatible with $\rho$, then $(H^i_I(M) \subset H^i_{IS}(N), \tau^i)$ is a differential direct summand compatible with $\rho$ for all $i$, where $\tau^i:=H^i_I(\tau):H^i_{IS}(N)=H^i_{I}(N) \to H^i_I(M)$.
\end{lemma}
\proof We adapt the relevant part of the proof of \cite[Theorem 3.8]{AHN17}, namely the argument for showing that $H^i_I(R)$ is a differential direct summand of $H^i_{IS}(S)$, to the setting of right $D$-modules. Choose generators $I=(f_1,\dots,f_r)$, and consider the \v{C}ech complexes $\check{\mathrm{C}}^{\bullet}(f_1,\dots,f_r;M)$ and $\check{\mathrm{C}}^{\bullet}(f_1,\dots,f_r;N)$ whose $i$th cohomology modules are given by $H^i_I(M)$ and $H^i_{IS}(N)$, respectively. By Lemma \ref{Proposition 3.6}, $\check{\mathrm{C}}^{\bullet}(f_1,\dots,f_r;M) \subset \check{\mathrm{C}}^{\bullet}(f_1,\dots,f_r;N)$ is a chain complex of differential direct summands compatible with $\rho$. The result then follows from Lemma \ref{Lemma 3.7}. \qed

\begin{remark}\label{polynomial extension}
Fix an integer $r\geq 1$. Given a direct summand $R \subset S$ of $\bbc$-algebras with splitting $\rho$, let $\rho'$ denote the map $\rho \otimes 1 : S \otimes_{\bbc} \bbc[t_1,\dots,t_r] \to R \otimes_{\bbc} \bbc[t_1,\dots,t_r]$. Then $R[t_1,\dots,t_r] \subset S[t_1,\dots,t_r]$ is a direct summand of $\bbc$-algebras with splitting $\rho'$. It is easy to see that for $Q = \sum_{\alpha,\beta\in \bbz_{\geq 0}^r} Q_{\alpha,\beta} t^{\alpha}\partial_t^{\beta} \in D_{S[t_1,\dots,t_r]}$ with $Q_{\alpha,\beta}\in D_S$ (recall \eqref{local product}), we have
\begin{equation}\label{restriction}
\rho' \circ Q|_{R[t_1,\dots,t_r]} =  \sum_{\alpha,\beta\in \bbz_{\geq 0}^r} (\rho \circ Q_{\alpha,\beta}|_R) t^{\alpha}\partial_t^{\beta} \in D_{R[t_1,\dots,t_r]}.
\end{equation}
Thus, given a differential direct summand $(M\subset N,\tau)$ compatible with $\rho$, if we denote by $\tau'$ the map $\tau \otimes 1 : N \otimes_{\bbc} \bbc[t_1,\dots,t_r] \to M \otimes_{\bbc} \bbc[t_1,\dots,t_r]$, then $(M[t_1,\dots,t_r] \subset N[t_1,\dots,t_r],\tau')$ is a differential direct summand compatible with $\rho'$.
\end{remark}

\begin{lemma}\label{direct summand b function}
Fix an integer $r\geq 1$ and denote the standard coordinates on $\bba^r$ by $t_1,\dots,t_r$. Given a direct summand $R \subset S$ of finite type $\bbc$-algebras with splitting $\rho$, let $\rho'$ denote the map $\rho \otimes 1 : S \otimes_{\bbc}\bbc[t_1,\dots,t_r] \to R \otimes_{\bbc}\bbc[t_1,\dots,t_r]$ and let $\ms$ and $\ns$ be quasicoherent right $\ds_{\Spec R} \otimes_{\bbc} D_{\bba^r}$- and right $\ds_{\Spec S} \otimes_{\bbc} D_{\bba^r}$-modules, respectively, such that $\Gamma(\Spec R,\ms)$ is an $R[t_1,\dots,t_r]$-submodule of $\Gamma(\Spec S,\ns)$. Suppose there exists an $R[t_1,\dots,t_r]$-module map $\tau:\Gamma(\Spec S,\ns)\to \Gamma(\Spec R,\ms)$ such that $(\Gamma(\Spec R,\ms) \subset \Gamma(\Spec S,\ns),\tau)$ is a differential direct summand compatible with $\rho'$. Then for every $v\in \Gamma(\Spec R,\ms)$, if we let $v^*$ denote the image of $v$ in $\Gamma(\Spec S,\ns)$, we have $b_v(s) \ | \ b_{v^*}(s)$.
\end{lemma}
\proof The proof follows the same idea as in \cite[Theorem 3.14]{AHN17} (see also \cite[Theorems 3.24 and 3.27]{AHJNTW22}); we include the details for completeness. By \eqref{b function affine local}, there exists $Q \in D_S \otimes_{\bbc} V^1D_{\bba^r}$ such that $v^*\cdot b_{v^*}(s) = v^*\cdot Q$. By \eqref{restriction}, we have $\rho' \circ b_{v^*}(s)|_{R[t_1,\dots,t_r]} = b_{v^*}(s)$ and $\rho \circ Q|_{R[t_1,\dots,t_r]} \in D_R \otimes_{\bbc} V^1D_{\bba^r}$, hence
$$
v \cdot b_{v^*}(s) = \tau(v^*\cdot b_{v^*}(s)) = \tau(v^*\cdot Q) = v \cdot (\rho \circ Q|_{R[t_1,\dots,t_r]}) \in v\cdot (D_R \otimes_{\bbc} V^1D_{\bba^r}).
$$
Thus $b_v(s) \ | \ b_{v^*}(s)$ by \eqref{b function affine local}. \qed

\section{The right $\ds$-module $B_{\bff}^{\omega}$}\label{bfomega}

In this section we introduce the right $\ds$-module analogue $B_{\bff}^{\omega}$ of $B_{\bff}$ on normal varieties.

Let $X$ be a normal variety. Given regular functions $f_1,\dots,f_r\in \os_X(X)$, we denote by $\bff$ the $r$-tuple $(f_1,\dots,f_r)$ and put $\Gamma_{\bff} = i_{\bff}(X) \subset X \times \bba^r$, where $i_{\bff}\colon  X \to X \times \bba^r, x \mapsto (x,\bff(x))$ is the graph embedding. We consider the local cohomology sheaf $\hs^r_{\Gamma_{\bff}}( \omega_{X\times \bba^r})$, which has a natural right $\ds_{X \times \bba^r}$-module structure (see, for example, \cite[Examples 2.1(iv)]{Lyu93} or \cite[page 109]{KS94}). We denote the standard coordinates on $\bba^r$ by $t_1,\dots,t_r$ and the projection $X\times \bba^r \to X$ by $p_X$, and put $\rs_X = p_{X,*}\ds_{X\times \bba^r}$. We define the right $\rs_X$-module $B^{\omega}_{\bff}$ as
$$
B^{\omega}_{\bff} :=
p_{X,*}(\hs^r_{\Gamma_{\bff}}( \omega_{X\times \bba^r})).
$$

\begin{lemma}\label{bfomega action}
Given a normal variety $X$ and regular functions $f_1,\dots,f_r\in \os_X(X)$, if we denote the $r$-tuple $(f_1,\dots,f_r)$ by $\bff$, then we have a natural isomorphism of $\os_X$-modules
$$
B_{\bff}^{\omega} \simeq \bigoplus_{\beta\in \bbz_{\geq 0}^r} \omega_X\partial_t^{\beta} \delta_{\bff}^{\omega}= \omega_X\otimes_{\bbc} \bbc[\partial_{t_1},\dots,\partial_{t_r}] \delta_{\bff}^{\omega},
$$
where the actions of $\os_X$ and $\partial_{t_i}$ are the obvious ones, while the actions of $\theta \in \sDer_{\bbc}(\os_X)$ and $t_i$ are given by 
$$
v\partial_t^{\beta}\delta^{\omega}_{\bff} \cdot \theta =  \theta(v)\partial_t^{\beta}\delta^{\omega}_{\bff} + \sum_{i=1}^r \theta(f_i) v \partial_t^{\beta+e_i}\delta^{\omega}_{\bff}
\ \ \
\text{and}
\ \ \
v\partial_t^{\beta}\delta^{\omega}_{\bff}  \cdot t_i = f_i v \partial_t^{\beta}\delta^{\omega}_{\bff} + \beta_i v \partial_t^{\beta-e_i}\delta^{\omega}_{\bff},
$$
where $e_1,\dots,e_r$ denotes the standard basis of $\bbz^r$. In particular, $B_{\bff}^{\omega}$ is $\os_X$-torsion-free.
\end{lemma}
\proof Let $e=(1,\dots,1)\in \bbz^r$, $dt=dt_1\wedge \cdots \wedge dt_r \in \omega_{\bba^r}$, and $(\bft - \bff)^{\alpha} = \prod_{i=1}^r (t_i-f_i)^{\alpha_i}$ for $\alpha \in \bbz^r$. Since the graph $\Gamma_{\bff} \subset X\times \bba^r$ is defined by $t_1-f_1,\dots,t_r-f_r$, we have
$$
\hs^r_{\Gamma_{\bff}}( \omega_{X\times \bba^r})=\omega_{X\times\bba^r}[1/(\bft-\bff)^e]/\sum_{i=1}^r \omega_{X\times \bba^r}[1/(\bft-\bff)^{e-e_i}].
$$
Since $p_{X,*}\omega_{X\times \bba^r}= \omega_X[t_1,\dots,t_r]dt$, we have the $\os_X$-module decomposition
$$
p_{X,*}(\omega_{X\times \bba^r}[1/(\bft-\bff)^e]) = \bigoplus_{\alpha\in \bbz^r} \omega_X (\bft-\bff)^{\alpha} dt,
$$
which induces the following isomorphism of $\os_X$-modules:
$$
B^{\omega}_{\bff} \simeq \bigoplus_{\beta \in \bbz_{\geq 0}^r} \omega_X \frac{1}{(\bft - \bff)^{\beta+e}} dt.
$$ 
Given $v\in \omega_X$, $\beta\in \bbz^r_{\geq 0}$, and $\theta\in \sDer_{\bbc}(\os_X)$, observe that the following three equations hold in $p_{X,*}(\omega_{X\times \bba^r}[1/(\bft-\bff)^e])$:
$$
\begin{aligned}
v\frac{1}{(\bft-\bff)^{e}}dt \cdot \partial_t^{\beta} &=  \beta! v  \frac{1}{(\bft - \bff)^{\beta+e}}dt \\
v\frac{1}{(\bft-\bff)^{\beta+e}}dt \cdot \theta &= \theta(v) \frac{1}{(\bft-\bff)^{\beta+e}}dt  + \sum_{i=1}^r v\theta(f_i) \beta_i \frac{1}{(\bft-\bff)^{\beta+e_i+e}}dt \\
v\frac{1}{(\bft-\bff)^{\beta+e}}dt \cdot t_i  &=  v  \frac{1}{(\bft - \bff)^{\beta-e_i+e}}dt+f_i v \frac{1}{(\bft - \bff)^{\beta+e}}dt.
\end{aligned}
$$
Indeed, the first, second, and third equations follow from (\ref{canonical dmod}), Remark \ref{right action formula}, and the fact that $t_i=(t_i-f_i)+f_i$, respectively. Therefore, if we write $v \partial_t^{\beta} \delta_{\bff}^{\omega} := \beta! v\frac{1}{(\bft-\bff)^{\beta+e}}dt  \in  B^{\omega}_{\bff}$ for $\beta\in \bbz_{\geq 0}^r$ and $v\in \omega_X$, the result follows. \qed

\begin{definition}\label{bfomega pullback}
Let $\pi: Y \to X$ be a morphism between normal varieties of the same pure dimension such that $\pi^{-1}(X_{sing})$ has codimension $\geq 2$, let $f_1,\dots,f_r\in \os_X(X)$ be regular functions, and let $\bff$ and $\pi^*\bff$ denote the $r$-tuples $(f_1,\dots,f_r)$ and $(\pi^*f_1,\dots,\pi^*f_r)$, respectively. Given a section $v=\sum_{\beta} v_{\beta} \partial_t^{\beta} \delta_{\bff}^{\omega}\in \Gamma(X,B_{\bff}^{\omega})$ with $v_{\beta}\in \Gamma(X,\omega_X)$, we define
$$
\pi^*v := \sum_{\beta} \pi^*(v_{\beta}) \partial_t^{\beta} \delta_{\pi^*\bff}^{\omega} \in \Gamma(Y,B^{\omega}_{\pi^*\bff}),
$$
where $\pi^*(v_{\beta})\in \Gamma(Y,\omega_Y)$ is given by the map $\pi^*$ defined in \eqref{pullback forms}.
\end{definition}


\begin{lemma}\label{left right switch}
Let $X$ be a smooth variety with global coordinates $x_1,\dots,x_n\in \os_X(X)$, let $f_1,\dots,f_r\in \os_X(X)$ be regular functions, and let $\bff$ denote the $r$-tuple $(f_1,\dots,f_r)$. Given a section $u=\sum_{\beta}u_{\beta} \partial_t^{\beta}\delta_{\bff}\in \Gamma(X, B_{\bff})$ with $u_{\beta}\in \os_X(X)$, if we put
$$
v=\sum_{\beta} (-1)^{|\beta|}(u_{\beta} dx_1\wedge \cdots \wedge dx_n) \partial_t^{\beta} \delta^{\omega}_{\bff} \in \Gamma(X, B^{\omega}_{\bff}),
$$
then we have $b_u(s) = b_v(-s-r)$. 
\end{lemma}
\proof Consider the $\bbc$-algebra involution $\FT: \rs_X\xrightarrow{\sim} \rs^{op}_X$ given by $\FT(\partial_{x_i}) = -\partial_{x_i}$, $\FT(\partial_{t_i}) = -\partial_{t_i}$, and $\FT(g)=g$ for $g\in \os_X$. Observe that
$$
\FT(s) = \sum_{i=1}^r t_i\partial_{t_i} = \sum_{i=1}^r (\partial_{t_i} t_i -1 ) = -s-r.
$$
If we view $B_{\bff}^{\omega}$ as a left $\rs_X$-module via restriction of scalars along $\FT$, then by \eqref{bf action} and Lemma \ref{bfomega action}, the following map is an isomorphism of left $\rs_X$-modules:
$$
\phi: B_{\bff} \xrightarrow{\sim} B^{\omega}_{\bff}, \sum_{\beta}g_{\beta} \partial_t^{\beta}\delta_{\bff} \mapsto \sum_{\beta} (-1)^{|\beta|}(g_{\beta}dx_1\wedge \cdots \wedge dx_n) \partial_t^{\beta}\delta^{\omega}_{\bff}.
$$
Since $\FT(V^k\rs_X) = V^k\rs_X^{op}$ by Lemma \ref{explicit V filtration}, $\phi$ induces the isomorphism
$$
V^0\rs_X \cdot u/ V^1\rs_X \cdot u \xrightarrow{\sim} v\cdot V^0\rs_X/ v\cdot V^1\rs_X.
$$
Thus $b_u(s) = \FT(b_v(s)) = b_v(\FT(s)) = b_v(-s-r)$. \qed

\section{Finite group quotients and direct summands of canonical modules}\label{direct summands of canonical modules}

Throughout this section, we fix a smooth affine variety $X=\Spec S$ of pure dimension, together with a finite group $G$ acting on $X$ from the right (equivalently, on $S$ from the left). Let $R=S^G$ denote the ring of invariants and set $Y=\Spec R$. We denote the morphism induced by the inclusion $R \subset S$ by $\pi: X \to Y$. By \cite[Corollary 6.4.6 and Remark 6.4.7]{BH93}, $Y$ is a Cohen-Macaulay variety. It is well-known that $\pi$ is finite and surjective and that $Y$ is normal.

Note that $R \subset S$ is a direct summand of $\bbc$-algebras with splitting given by 
$$
\rho=\frac{1}{|G|}\sum_{g\in G} g: S \to R.
$$
Consider the left $G$-action on $\Gamma(X,\omega_X)$ given by pulling back differential forms and the map
$$
\tau=\frac{1}{|G|}\sum_{g\in G} g: \Gamma(X,\omega_X) \to \Gamma(X,\omega_X)^G.
$$
By \cite[Theorem 2.7]{Pes84}, the map $\pi^*: \Gamma(Y,\omega_Y) \to \Gamma(X,\omega_X)$ defined in \eqref{pullback forms} is injective and maps onto $\Gamma(X,\omega_X)^G$, so that we have the $R$-module isomorphism
\begin{equation}\label{image onto invariants}
\pi^*: \Gamma(Y,\omega_Y) \xrightarrow{\sim} \Gamma(X,\omega_X)^G.
\end{equation}
Since $\Gamma(Y,\omega_Y)$ is a right $D_R$-module, \eqref{image onto invariants} induces a right $D_R$-module structure on $\Gamma(X,\omega_X)^G$.

The following theorem is the key technical result of this section.

\begin{theorem}\label{canonical direct summand}
The pair $(\Gamma(X,\omega_X)^G \subset \Gamma(X,\omega_X), \tau)$ is a differential direct summand compatible with $\rho$, where the right $D_R$-module structure on $\Gamma(X,\omega_X)^G$ is given by (\ref{image onto invariants}). 
\end{theorem}

We note that the left $D$-module version of theorem \ref{canonical direct summand}, which says that $(R\subset S,\rho)$ is a differential direct summand compatible with $\rho$, is immediate from the definition.

Before proving the theorem, we record its main consequence for $b$-functions.

\begin{corollary}\label{finite group b function}
Let $f_1,\dots,f_r\in \os_Y(Y)$ be regular functions and let $\bff$ denote the $r$-tuple $(f_1,\dots,f_r)$. Given a section $v \in \Gamma(Y,B_{\bff}^{\omega})$, we have $b_v(s) \ | \ b_{\pi^*v}(s)$, where $\pi^*v\in \Gamma(X, B_{\pi^*\bff}^{\omega})$ is defined as in Definition \ref{bfomega pullback}.
\end{corollary}
\proof Let $M = \Gamma(X,\omega_X)^G$ and $N=\Gamma(X,\omega_X)$, and let $\rho'$ and $\tau'$ denote the induced maps $\rho \otimes 1: S[t_1,\dots,t_r] \to R[t_1,\dots,t_r]$ and $\tau\otimes 1: N[t_1,\dots,t_r] \to M[t_1,\dots,t_r]$, respectively. By Theorem \ref{canonical direct summand} and Remark \ref{polynomial extension}, $(M[t_1,\dots,t_r] \subset N[t_1,\dots,t_r], \tau')$ is a differential direct summand compatible with $\rho'$. By Lemma \ref{local coh direct summand}, if we denote by $I\subset R[t_1,\dots,t_r]$ the ideal generated by $t_1-f_1,\dots,t_r-f_r$, then
$$
(H^r_I(M[t_1,\dots,t_r]) \subset H^r_{IS[t_1,\dots,t_r]}(N[t_1,\dots,t_r]), H^r_I(\tau'))
$$
is a differential direct summand compatible with $\rho'$. Since 
$$
\Gamma(Y,B_{\bff}^{\omega})= H^r_I(\Gamma(Y,\omega_Y)[t_1,\dots,t_r]) \simeq H^r_I(M[t_1,\dots,t_r]) 
$$
and $\Gamma(X,B_{\pi^*\bff}^{\omega})=H^r_{IS[t_1,\dots,t_r]}(N[t_1,\dots,t_r])$, the result follows from Lemma \ref{direct summand b function}. \qed

The remainder of this section is devoted to the proof of Theorem \ref{canonical direct summand}. We begin with some discussion and establish two auxiliary lemmas that will be used in the proof.

For each $g\in G$, we denote the automorphism of $X$ induced by the action of $g$ by
$$
r_g: X \to X, x \mapsto x \cdot g.
$$
Let $U \subset X$ be an open subset. For a regular function $f\in \os_X(U)$, let 
\begin{equation}\label{function action}
gf := r_g^*f \in \os_X(r_g^{-1}(U))
\end{equation}
denote the regular function obtained by pulling $f$ back along $r_g: r_g^{-1}(U) \to U$. For a differential form $v\in \Gamma(U,\omega_X)$, let 
\begin{equation}\label{form action}
g \cdot v :=r_g^*v \in \Gamma(r_g^{-1}(U), \omega_X)
\end{equation}
denote the differential form obtained by pulling $v$ back along $r_g: r_g^{-1}(U) \to U$. Finally, for a differential operator $P\in \Gamma(U,\ds_X)$, let 
\begin{equation}\label{operator action}
g \cdot P \in \Gamma(r_g^{-1}(U), \ds_X)
\end{equation}
denote the differential operator given by $(g \cdot P)(f) = gP(g^{-1}f)$ for every $f\in \os_X(V)$ and open subset $V \subset r_g^{-1}(U)$. 

Observe that $G$ acts on $\pi_*\os_X$, $\pi_*\omega_X$, and $\pi_*\ds_X$ on the left via \eqref{function action}, \eqref{form action}, and \eqref{operator action}, respectively. Note that $(\pi_*\os_X)^G=\os_Y$ and the inclusion $\pi_*\os_X \into \pi_*\ds_X$ is $G$-equivariant. Also, $G$ acts on $\pi_*\ds_X$ by $\bbc$-algebra automorphisms.

\begin{remark}\label{global coordinate action}
Let $U \subset X$ be an open subset with global coordinates $x_1,\dots,x_n \in \os_X(U)$, and fix an element $g\in G$. Note that $gx_1,\dots,gx_n \in \os_X(r_g^{-1}(U))$ are global coordinates on $r_g^{-1}(U)$. By the chain rule, it is easy to see that $g \cdot \partial_{x_i} = \partial_{gx_i}$. Thus $g \cdot \partial_x^{\alpha} = \partial_{gx}^{\alpha}$ for $\alpha \in \bbz_{\geq 0}^n$. 
\end{remark}

We omit the proof of the following easy lemma.

\begin{lemma}\label{G equivariance}
Given $g\in G$, $v \in \pi_*\omega_X$, and $P \in \pi_*\ds_X$, we have $g \cdot (v \cdot P) = (g \cdot v) \cdot (g \cdot P)$.
\end{lemma}

We define a $\bbc$-algebra morphism
\begin{equation}\label{res}
\res: (\pi_*\ds_X)^G \to \ds_Y
\end{equation}
as follows. Let $U \subset Y$ be an open subset and let $\pi_U:=\pi|_{\pi^{-1}(U)}: \pi^{-1}(U) \to U$. Given $P\in \Gamma(U,(\pi_*\ds_X)^G)$, we may view $P$ as a $\bbc$-linear endomorphism of $\os_X$. Then $\pi_{U,*}P$ is a $G$-invariant endomorphism of $\pi_{U,*}\os_{\pi^{-1}(U)} = (\pi_*\os_X)|_U$, hence it restricts to a $\bbc$-linear endomorphism of $((\pi_*\os_X)|_U)^G=\os_U$. We define $\res(P)\in \Gamma(U,\sEnd_{\bbc}(\os_Y))$ to be this restriction. Since $P$ is a differential operator, it easily follows that $\res(P)\in \Gamma(U,\ds_Y)$, completing the definition of \eqref{res}. 

We view the canonical sheaf $\omega_Y$ as a right $(\pi_*\ds_X)^G$-module via restriction of scalars along $\res: (\pi_*\ds_X)^G \to \ds_Y$. Note that the right $\pi_*\ds_X$-module structure on $\pi_*\omega_X$ induces a right $(\pi_*\ds_X)^G$-module structure on $(\pi_*\omega_X)^G$ by Lemma \ref{G equivariance}.

By \cite[Theorem 2.7]{Pes84}, the natural morphism $\pi^*: \omega_Y \to \pi_*\omega_X$ is injective and maps onto $(\pi_*\omega_X)^G$, so that we have the $\os_Y$-module isomorphism
\begin{equation}\label{peskin}
\pi^*: \omega_Y \xrightarrow{\sim} (\pi_*\omega_X)^G.
\end{equation}

\begin{lemma}\label{first and second action}
The map $\pi^*: \omega_Y \to (\pi_*\omega_X)^G$ is $(\pi_*\ds_X)^G$-linear, that is, for every $v \in \omega_Y$ and $Q \in (\pi_*\ds_X)^G$, we have $\pi^*(v \cdot \res(Q)) = (\pi^*v) \cdot Q$.
\end{lemma}
\noindent{\it Proof}. We need to show that the following diagram commutes:
\begin{equation}\label{commutative diagram}
\begin{tikzcd}
\omega_Y \otimes_{\os_Y} (\pi_*\ds_X)^G \arrow[r]\arrow[d,"\pi^* \otimes 1"]& \omega_Y \arrow[d,"\pi^*"] \\
(\pi_*\omega_X)^G \otimes_{\os_Y} (\pi_*\ds_X)^G \arrow[r] & (\pi_*\omega_X)^G.
\end{tikzcd}
\end{equation}
By generic smoothness, there exists a smooth dense open subset $U$ of $Y$ such that the morphism $\pi|_{\pi^{-1}(U)}\colon \pi^{-1}(U) \to U$ is \'{e}tale. Let $j: U \to Y$ denote the inclusion. Since $(\pi_*\omega_X)^G \simeq \omega_Y$ by \eqref{peskin}, the natural morphism $(\pi_*\omega_X)^G \to j_*j^*(\pi_*\omega_X)^G$ is injective, hence it suffices to show that the diagram commutes after restricting to $U$. Thus we may assume that $Y$ is smooth and $\pi: X \to Y$ is \'{e}tale. Since we can check commutativity of the diagram locally on $Y$, we may further assume that $Y$ has global coordinates $y_1,\dots,y_n \in \os_Y(Y)$.

If we put $x_i=\pi^*y_i$, then $x_1,\dots,x_n \in \os_X(X)$ are global coordinates on $X$ since $\pi$ is \'{e}tale. Note that $gx_i=x_i$ for each $i$, so by Remark \ref{global coordinate action}, we have $g \cdot \partial_{x_i} = \partial_{x_i}$, hence $\partial_{x_i}\in (\pi_*\ds_X)^G$. It is then easy to see that
$$
(\pi_*\ds_X)^G = \bigoplus_{\alpha \in \bbz_{\geq 0}^n} (\pi_*\os_X)^G\partial_x^{\alpha} = \bigoplus_{\alpha \in \bbz_{\geq 0}^n} \os_Y \partial_x^{\alpha}.
$$
By the chain rule, we have $\res(\partial_{x_i}) = \partial_{y_i}$, so that $\res(\partial_x^{\alpha}) = \partial_y^{\alpha}$ for $\alpha \in \bbz_{\geq 0}^n$. To show that \eqref{commutative diagram} commutes, it suffices to show that 
$$
\pi^* (f dy_1 \wedge \cdots \wedge dy_n \cdot \res(\partial_x^{\alpha})) 
=
\pi^*(fdy_1 \wedge \cdots \wedge dy_n) \cdot \partial_x^{\alpha}
$$
for $f\in \os_Y$ and $\alpha \in \bbz_{\geq 0}^n$. This follows since both sides of the equation are equal to $(-1)^{|\alpha|} \partial^{\alpha}_x(\pi^*(f))dx_1\wedge \cdots \wedge dx_n$ by \eqref{canonical dmod}. \qed

\noindent{\it Proof of Theorem \ref{canonical direct summand}}. Let $v\in \Gamma(Y,\omega_Y)$ and $Q \in D_S$. Since the right $D_R$-module structure on $\Gamma(X,\omega_X)^G$ is given by (\ref{image onto invariants}), it suffices to show that
$$
\tau(\pi^*v \cdot Q) = \pi^*(v \cdot (\rho \circ Q|_R)).
$$
Consider the following $G$-invariant differential operator of $S$:
$$
P = \frac{1}{|G|}\sum_{g\in G} g \cdot Q \in D^G_S = \Gamma(Y, (\pi_*\ds_X)^G).
$$
For $f\in R$, observe that 
$$
(\rho \circ Q|_R)(f) = \frac{1}{|G|}\sum_{g\in G} g (Q(f)) = \frac{1}{|G|}\sum_{g\in G} g Q(g^{-1}f) = \res(P)(f),
$$
where the second equality uses the fact that $f\in S^G$. Thus $\rho \circ Q|_R=\res(P)$, so by Lemma \ref{first and second action}, we have
$$
\pi^*(v \cdot (\rho \circ Q|_R)) = \pi^*(v \cdot\res(P)) = (\pi^*v) \cdot P= \frac{1}{|G|}\sum_{g\in G} \pi^*v \cdot (g \cdot Q).
$$
On the other hand, by Lemma \ref{G equivariance}, we have
$$
\tau(\pi^*v \cdot Q) = \frac{1}{|G|}\sum_{g\in G} g \cdot (\pi^*v \cdot Q)
= \frac{1}{|G|}\sum_{g\in G} (g \cdot \pi^*v) \cdot (g \cdot Q)
= \frac{1}{|G|}\sum_{g\in G} \pi^*v \cdot (g \cdot Q),
$$
where the third equality uses the fact that $\pi^*v \in \Gamma(X,\omega_X)^G$. Thus we are done. \qed

\section{\'{E}tale pullback}\label{etale pullback}

The goal of this section is to show that $b$-functions of sections of $B_{\bff}^{\omega}$ are invariant under surjective \'{e}tale pullback.

\begin{theorem}\label{etale pull b}
Let $\pi: X \to Y$ be a surjective \'{e}tale morphism of varieties of the same pure dimension such that $Y$ is normal, let $f_1,\dots,f_r\in \os_Y(Y)$ be regular functions, and let $\bff$ denote the $r$-tuple $(f_1,\dots,f_r)$. Given a section $v\in \Gamma(Y, B_{\bff}^{\omega})$, we have $b_v(s) = b_{\pi^*v}(s)$, where $\pi^*v\in \Gamma(X, B_{\pi^*\bff}^{\omega})$ is defined as in Definition \ref{bfomega pullback}.
\end{theorem}

The proof of Theorem \ref{etale pull b} will be given at the end of this section. We begin with some discussion and two auxiliary lemmas that will be used in the proof.

Given a map of $\bbc$-algebras $\phi: R \to S$ and a differential operator $P\in D_R$, we say that a differential operator $Q\in D_S$ extends $P$ if $\phi(P(x)) = \widetilde{P}(\phi(x))$ for every $x\in R$.

\begin{lemma}[{\cite[Theorems 2.2.5 and 2.2.10]{Mas91}}]\label{etale unique extension}
Let $\phi: R \to S$ be an \'{e}tale map of finitely generated $\bbc$-algebras. Given a differential operator $P\in D_R$, there exists a unique differential operator $\widetilde{P}\in D_S$ that extends $P$. Moreover, the map $D_R \to D_S$ that sends $P\in D_R$ to its unique extension is a ring map such that the diagram
\[
\begin{tikzcd}
R \arrow[r]\arrow[d] & S \arrow[d] \\
D_R \arrow[r] & D_S.
\end{tikzcd}
\]
commutes and the induced left $S$-module map $S \otimes_R D_R \to D_S$ is an isomorphism.
\end{lemma}

Given an \'{e}tale morphism $\pi: X \to Y$ of varieties, we define a $\bbc$-algebra map $\ds_Y \to \pi_*\ds_X$ as follows. Over an open subset $W \subset Y$, we send $P\in \Gamma(W,\ds_Y)$ to the unique element $\widetilde{P}\in \Gamma(\pi^{-1}(W),\ds_X)$ such that $\widetilde{P}|_U$ extends $P|_V$ for every affine open $U\subset \pi^{-1}(W)$ that maps into another affine open $V \subset W$. The existence and uniqueness of $\widetilde{P}$ follow from Lemma \ref{etale unique extension}. By the same lemma, if we consider the corresponding $\bbc$-algebra map $\pi^{-1}\ds_Y\to \ds_X$, then the diagram
\[
\begin{tikzcd}
\pi^{-1}\os_Y \arrow[r]\arrow[d] & \os_X \arrow[d] \\
\pi^{-1}\ds_Y \arrow[r] & \ds_X.
\end{tikzcd}
\]
commutes and induces the left $\os_X$-module isomorphism $\os_X \otimes_{\pi^{-1}\os_Y} \pi^{-1}\ds_Y \xrightarrow{\sim} \ds_X$.

\begin{lemma}\label{right pull diff}
Given an \'{e}tale morphism $\pi: X \to Y$ of varieties such that $Y$ is normal, the induced morphism $\pi^{-1}\ds_Y \otimes_{\pi^{-1}\os_Y} \os_X \to \ds_X$ of right $\os_X$-modules is an isomorphism.
\end{lemma}
\proof Let $j: Y_{sm} \to Y$ and $j': X_{sm} \to X$ denote the inclusions, and let $\pi'$ denote the morphism $\pi|_{X_{sm}}: X_{sm} \to Y_{sm}$. Since $\pi$ is \'{e}tale, $X$ is normal and $\pi^{-1}(Y_{sm})=X_{sm}$, so by \eqref{pull push normal isom} we have $\ds_X = j'_*\ds_{X_{sm}}$ and
$$
\pi^{-1}\ds_Y \otimes_{\pi^{-1}\os_Y} \os_X \xrightarrow{\sim}
\pi^{-1}(j_*\ds_{Y_{sm}})\otimes_{\pi^{-1}\os_Y} \os_X \xrightarrow{\sim}
j'_*(\pi'^{-1}\ds_{Y_{sm}} \otimes_{\pi'^{-1}\os_{Y_{sm}}} \os_{X_{sm}} ),
$$
where the second isomorphism follows from flat base change.
Thus, we reduce to the case when $Y$ is smooth. Since the problem is local on $X$ and on $Y$, we may further reduce to the case when $X$ and $Y$ are affine and $Y$ has global coordinates $y_1,\dots,y_n\in \os_Y(Y)$. Let $X=\Spec S$ and $Y=\Spec R$, and consider the $\bbc$-algebra map $D_R \to D_S$ defined in Lemma \ref{etale unique extension}. It suffices to show that the induced right $S$-module map $D_R \otimes_R S \to D_S$ is an isomorphism. If we put $x_i=\pi^*y_i$ for each $i$, then $x_1,\dots,x_n\in \os_X(X)$ are global coordinates on $X$ since $\pi$ is \'{e}tale, and by the chain rule, $\partial_{x_i}\in D_S$ extends $\partial_{y_i}\in D_R$. Thus the map $D_R \to D_S$ sends $\partial_y^{\alpha}$ to $\partial_x^{\alpha}$ for $\alpha \in \bbz_{\geq 0}^n$. Since $D_R$ is a free right $R$-module with basis $\{\partial_y^{\alpha}\ | \ \alpha \in \bbz_{\geq 0}^n\}$ and $D_S$ is a free right $S$-module with basis $\{\partial_x^{\alpha}\ | \ \alpha \in \bbz_{\geq 0}^n\}$ (see, for example, \cite[Lemma 1.2.7]{HTT08}), the map $D_R \otimes_R S \to D_S$ is an isomorphism, as desired. \qed

\noindent{\it Proof of Theorem \ref{etale pull b}}. By flat base change, we have $\pi^*\omega_Y \simeq \omega_X$, hence by Lemma \ref{bfomega action}, we have $\pi^*B_{\bff}^{\omega} \simeq B_{\pi^*\bff}^{\omega}$. By Lemma \ref{right pull diff}, we have
$$
\pi^{-1}(V^k\rs_Y) \otimes_{\pi^{-1}\os_Y} \os_X =  (\pi^{-1}\ds_Y \otimes_{\pi^{-1}\os_Y} \os_X) \otimes_{\bbc} V^kD_{\bba^r} \simeq \ds_X \otimes_{\bbc}V^kD_{\bba^r} = V^k\rs_X
$$
for $k\in \bbz$. Since $\pi$ is flat, we then obtain $\pi^*(v\cdot V^k\rs_Y) = \pi^*v \cdot V^k\rs_X \subset B_{\pi^*\bff}^{\omega}$ and
$$
\pi^*(v\cdot V^0\rs_Y/v\cdot V^1\rs_Y) = \pi^*v \cdot V^0\rs_X/\pi^*v\cdot V^1\rs_X.
$$
If we denote the actions of $s$ on $v\cdot V^0\rs_Y/v\cdot V^1\rs_Y$ and $\pi^*(v\cdot V^0\rs_Y/v\cdot V^1\rs_Y)$ by $s_Y$ and $s_X$, respectively, then $s_Y$ is $\os_Y$-linear and $\pi^*s_Y=s_X$. By Remark \ref{b function annihilator}, $b_v(s)$ and $b_{\pi^*v}(s)$ are the minimal polynomials of $s_Y$ and $s_X$, respectively. The result then follows since $\pi$ is faithfully flat. \qed

\section{Coherent modules and good filtrations}\label{coherent modules and good filtrations}

Throughout this section, we fix an integer $r \geq 1$ and a variety $X$ such that the morphism $\rSpec_X(\gr^F_{\bullet}\ds_X) \to X$ is of finite type. We denote the standard coordinates on $\bba^r$ by $t_1,\dots,t_r$ and put $\rs_X = \ds_X \otimes_{\bbc} D_{\bba^r}$. Recall the $V$-filtrations on $\ds_{X\times \bba^r}$ and $\rs_X$ from Section \ref{the v-filtration}. 

We say that a graded $R^+_V\rs_X$-module is coherent if it is coherent as an $R^+_V\rs_X$-module. Given a graded $R^+_V\rs_X$-module $\ns=\bigoplus_{k\in \bbz}\ns_k T^k$, we say that an $\os_X$-submodule $\gs \subset \ns$ is an $\os_X$-coherent graded subsheaf if $\gs$ is $\os_X$-coherent and $\gs = \bigoplus_{k\in \bbz} \gs_kT^k$.

Consider the following characterization of coherent modules over certain sheaves of rings.

\begin{lemma}\label{coherent char}
Let $\as$ be a quasicoherent sheaf of rings on a scheme $Z$ such that $\Gamma(U,\as)$ is a right noetherian ring for every affine open $U \subset Z$. Then a right $\as$-module $\ns$ is coherent if and only if it is quasicoherent and $\Gamma(U,\ns)$ is a finitely generated $\Gamma(U,\as)$-module for every affine open $U \subset Z$. The corresponding statement for left modules also holds, with the obvious modifications.
\end{lemma}
\proof The proof is identical to that of \cite[Proposition 1.4.9]{HTT08}, so we omit it. \qed 

The next lemma shows that the quasicoherent sheaves of rings $V^0\ds_{X\times \bba^r}$, $V^0\rs_X$, and $R^+_V\rs_X$ satisfy the hypothesis of Lemma \ref{coherent char}. Recall that a noetherian ring is, by definition, both left noetherian and right noetherian.

\begin{lemma}\label{finite type assumption}
The rings $\Gamma(U', V^0\ds_{X\times \bba^r})$, $\Gamma(U, V^0\rs_X)$, and $\Gamma(U, R^+_V\rs_X)$ are noetherian for affine open subsets $U' \subset X \times \bba^r$ and $U\subset X$.
\end{lemma}
\proof We begin with the following claim. Consider the filtrations on $V^kD_{\bba^r}$ and $R^+_VD_{\bba^r}$ given by $F_pV^kD_{\bba^r} = F_pD_{\bba^r} \cap V^kD_{\bba^r}$ and $F_pR^+_VD_{\bba^r} = \bigoplus_{k \geq 0} F_p V^kD_{\bba^r}T^k$ for $p\in \bbz$.

\noindent{\bf Claim}.
The $\bbc$-algebras $\gr_{\bullet}^F V^0D_{\bba^r}$ and $\gr_{\bullet}^F R^+_VD_{\bba^r}$ are of finite type.

\noindent{\it Proof of the claim}. Since $\gr_{\bullet}^F V^0D_{\bba^r}$ is a quotient of $\gr_{\bullet}^F R^+_VD_{\bba^r}$, it suffices to show that $\gr_{\bullet}^F R^+_VD_{\bba^r}$ is of finite type. Let $y_i \in \gr^F_1D_{\bba^r}$ denote the image of $\partial_{t_i}$, so that $\gr^F_{\bullet}D_{\bba^r} = \bbc[t_1,\dots,t_r, y_1,\dots,y_r]$. Then we have
$$
\gr_{\bullet}^F R^+_VD_{\bba^r} = \bbc[t^{\alpha}y^{\beta} T^k \ | \  |\alpha|-|\beta|\geq k] \subset ( \gr^F_{\bullet}D_{\bba^r})[T],
$$
which is a $\bbc$-algebra of finite type by Gordan's lemma (see \cite[Proposition 1.2.17]{CLS11}). \qed

Next, we consider the filtration on $V^0\ds_{X\times \bba^r} = \ds_X \boxtimes V^0\ds_{\bba^r}$ given by 
$$
F_kV^0\ds_{X\times \bba^r} = \sum_{p+q=k}F_p\ds_X \boxtimes F_qV^0\ds_{\bba^r},
$$
where $F_qV^0\ds_{\bba^r} = F_q\ds_{\bba^r} \cap V^0\ds_{\bba^r}$. Since $\gr^F_{\bullet}V^0\ds_{X \times \bba^r} = \gr^F_{\bullet}\ds_X \boxtimes \gr^F_{\bullet}V^0\ds_{\bba^r}$, we have
$$
\rSpec_{X\times \bba^r} (\gr^F_{\bullet}V^0\ds_{X \times \bba^r}) 
= \rSpec_X(\gr^F_{\bullet}\ds_X) \times \rSpec_{\bba^r}(\gr^F_{\bullet}V^0\ds_{\bba^r}).
$$
Thus the morphism $\rSpec_{X\times \bba^r} (\gr^F_{\bullet}V^0\ds_{X \times \bba^r}) \to X \times \bba^r$ is of finite type by our assumption on $X$ and the claim. If we put $F_k\Gamma(W, V^0\ds_{X \times \bba^r}) = \Gamma(W, F_kV^0\ds_{X \times \bba^r})$ for an affine open $W \subset X \times \bba^r$, then the $\bbc$-algebra $\gr^F_{\bullet}\Gamma(W, V^0\ds_{X \times \bba^r}) = \Gamma(W, \gr^F_{\bullet}V^0\ds_{X \times \bba^r})$ is of finite type, hence $\Gamma(W, V^0\ds_{X \times \bba^r})$ is noetherian by \cite[Proposition D.1.4]{HTT08}. 

Now consider the filtration on $R^+_V\rs_X$ given by
$$
F_kR^+_V\rs_X = \sum_{p+q=k} F_p\ds_X \otimes_{\bbc} F_qR^+_VD_{\bba^r}.
$$
Note that $\gr_F^{\bullet}(R^+_V\rs_X) = \gr^F_{\bullet}\ds_X \otimes_{\bbc} \gr^F_{\bullet} R^+_VD_{\bba^r}$. Thus, given an affine open $U \subset X$, if we put $F_k\Gamma(U, R^+_V\rs_X)=\Gamma(U,F_kR^+_V\rs_X )$, then the $\bbc$-algebra
$$
\gr^F_{\bullet}\Gamma(U, R^+_V\rs_X) = \Gamma(U, \gr^F_{\bullet}R^+_V\rs_X)  = \Gamma(U,  \gr^F_{\bullet}\ds_X) \otimes_{\bbc} \gr^F_{\bullet} R^+_VD_{\bba^r}
$$
is of finite type by our assumption on $X$ and the claim, hence $\Gamma(U, R^+_V\rs_X)$ is noetherian by \cite[Proposition D.1.4]{HTT08}. Finally, $\Gamma(U, V^0\rs_X)$ is noetherian since it is a quotient of $\Gamma(U, R^+_V\rs_X)$. \qed

Therefore, Lemma \ref{coherent char} provides a characterization of coherent $V^0\ds_{X\times \bba^r}$-, $V^0\rs_X$-, and $R^+_V\ds_{X\times \bba^r}$-modules.

\begin{lemma}\label{surjection}
Given a coherent right $V^0\ds_{X\times \bba^r}$-module $\ms$, there exists an $\os_{X \times \bba^r}$-coherent subsheaf $\fs \subset \ms$ that generates $\ms$ as a $V^0\ds_{X\times \bba^r}$-module.
\end{lemma}
\proof The argument is identical to that of \cite[Corollary 1.4.17]{HTT08}, so we omit it.\qed


\begin{corollary}\label{torsion}
Given a coherent right $V^0\ds_{X\times \bba^r}$-module $\ms$ supported on the closed subset $X\times 0 \subset X\times \bba^r$, there exists an integer $N \geq 0$ such that $\ms \cdot V^N\ds_{X\times \bba^r}=0$.
\end{corollary}
\proof Let $\is \subset \os_{X\times \bba^r}$ be the coherent sheaf of ideals generated by $t_1,\dots,t_r$. By Lemma \ref{surjection}, there exists a $\os_{X\times \bba^r}$-coherent subsheaf $\fs \subset \ms$ such that $\ms=\fs \cdot V^0\ds_{X\times \bba^r}$. Since $\fs$ is supported on $X\times 0$, there exists $N\geq 0$ such that $\fs \cdot \is^N = 0$. By Lemma \ref{generation from zero}, we then have
$$
\ms \cdot V^N\ds_{X\times \bba^r} = \fs \cdot V^N\ds_{X\times \bba^r}  = \fs \cdot \is^N \cdot V^0\ds_{X\times \bba^r} = 0. \qed
$$

\begin{lemma}\label{coherent pieces}
If $\ns=\bigoplus_{k\in \bbz} \ns_k T^k$ is a coherent graded right $R^+_V\rs_X$-module, then $\ns_k$ is a coherent right $V^0\rs_X$-module for all $k\in \bbz$.
\end{lemma}
\proof By Remark \ref{qcoh direct summands}, $\ns_k$ is quasicoherent for each $k$. Thus by Lemma \ref{coherent char}, it suffices to show that if $U \subset X$ is an affine open, then $\Gamma(U,\ns_k)$ is a finitely generated $\Gamma(U,V^0\rs_X)$-module for each $k$. By Lemma \ref{coherent char}, we can choose finitely many homogeneous $\Gamma(U,R^+_V\rs_X)$-module generators $v_1T^{k_1},\dots,v_mT^{k_m}$ of $\Gamma(U,\ns)=\bigoplus_{k\in \bbz} \Gamma(U,\ns_k) T^k$. Observe that
$$
\Gamma(U,\ns_k) = \sum_{k_i\leq k} v_i \cdot \Gamma(U,V^{k-k_i}\rs_X)
=
\sum_{k_i\leq k}\sum_{\substack{\alpha\in \bbz^r_{\geq 0}, \\ |\alpha|=k-k_i}} (v_i \cdot t^{\alpha} ) \cdot \Gamma(U,V^0\rs_X),
$$
where the first sum runs over $i\in \{1,\dots,m\}$ such that $k_i\leq k$, and the second equality follows from \eqref{push generation from zero}. Thus $\Gamma(U,\ns_k)$ is finitely generated as a $\Gamma(U,V^0\rs_X)$-module, as desired. \qed

\begin{lemma}\label{graded surjection}
Given a coherent graded right $R^+_V\rs_X$-module $\ns=\bigoplus_{k\in \bbz}\ns_kT^k$, there exists an $\os_X$-coherent graded subsheaf $\gs \subset \ns$ that generates $\ns$ as an $R^+_V\rs_X$-module.
\end{lemma}
\proof Let $X = \bigcup_i U_i$ be a finite affine open cover. For each $i$, choose a finite set $\{v_{ij}T^{k_{ij}}\}_j$ of homogeneous $\Gamma(U_i,R^+_V\rs_X)$-module generators of $\Gamma(U_i,\ns)$, and let $\gs_i = \bigoplus_k \gs_{ik} T^k \subset \ns|_{U_i}$ be the $\os_{U_i}$-coherent graded subsheaf generated by $\{v_{ij}T^{k_{ij}}\}_j$. Since $\ns_k$ is quasicoherent by Remark \ref{qcoh direct summands}, we can extend $\gs_{ik}$ to an $\os_X$-coherent subsheaf $\widetilde{\gs}_{ik}$ of $\ns_k$ by \cite[Lemma 28.23.2]{Stacks}. If we put $\widetilde{\gs}_i = \bigoplus_k \widetilde{\gs}_{ik}T^k  \subset \ns$, then $\gs= \sum_i \gs_i \subset \ns$ is an $\os_X$-coherent graded subsheaf such that $\ns=\gs \cdot R^+_V\rs_X$. \qed

\begin{definition}
Let $\ns$ be a right $V^0\rs_X$-module. A good filtration $U^{\bullet}\ns=(U^k\ns)_k$ of $\ns$ is a $\bbz_{\geq 0}$-indexed decreasing filtration of $\ns$ by $V^0\rs_X$-submodules $U^k\ns \subset \ns$ such that the following properties hold:
\begin{enumerate}
\item[i)] $U^0\ns = \ns$.
\item[ii)] $U^i\ns \cdot V^j\rs_X \subset U^{i+j}\ns$ for $i,j\geq 0$.
\item[iii)] $\bigoplus_{k\geq 0} U^k\ns T^k$ is a coherent $R^+_V\rs_X$-module.
\end{enumerate}
\end{definition}

\begin{remark}\label{good coherence}
If $U^{\bullet}\ns$ is a good filtration of a right $V^0\rs_X$-module $\ns$, then by Lemma \ref{coherent pieces}, $U^k\ns$ is a coherent $V^0\rs_X$-module for all $k\geq 0$. In particular, $\ns=U^0\ns$ is a coherent $V^0\rs_X$-module.
\end{remark}

The proof of the following lemma is standard, so we omit it.

\begin{lemma}\label{artin rees}
Given a good filtration $U^{\bullet}\ns$ of a coherent right $V^0\rs_X$-module $\ns$, there exists a constant $c \in \bbz_{\geq 0}$ such that $U^{k + c}\ns \subset \ns \cdot V^{k}\rs_X$ for every $k\geq 0$.
\end{lemma}

\section{Transfer bimodules and direct image}\label{transfer bimodule and direct image}

Throughout this section, we fix an integer $r\geq 1$ and a morphism $\pi: Y \to X$ of varieties such that $X=\Spec A$ is a smooth affine variety. We set $\rs_Y = \ds_Y \otimes_{\bbc} D_{\bba^r}$ and $\rs_X = \ds_X \otimes_{\bbc} D_{\bba^r}$. Recall the $V$-filtrations on $\ds_{Y\times \bba^r}$, $\ds_{X\times \bba^r}$, $\rs_Y$, and $\rs_X$ from Section \ref{the v-filtration}.

As in the case when $Y$ is smooth (see \cite[Definition 1.3.1]{HTT08}), we define the transfer bimodule of $\pi$ as
$$
\ds_{Y \to X} := \pi^*\ds_X=\os_Y \otimes_{\pi^{-1}\os_X} \pi^{-1}\ds_X.
$$
We recall that $\ds_{Y \to X}$ has a natural $(\ds_Y,\pi^{-1}\ds_X)$-bimodule structure (see also \cite[Theorem/Definition 3.2]{Jia23}). The right $\pi^{-1}\ds_X$-action is simply given by multiplication on the right tensor factor. The left $\ds_Y$-action is given as follows. Given an affine open $U=\Spec B \subset Y$, we have 
\begin{equation}\label{SV isom}
\Gamma(U, \ds_{Y\to X})=B\otimes_A D_A = D_{\bbc}(A,B)
\end{equation}
by \cite[Proposition 2.2.1]{SV97}, where $D_{\bbc}(A,B)$ is the module of $\bbc$-linear differential operators from $A$ to $B$. The module $D_{\bbc}(A,B)$ has a natural $(D_B,D_A)$-bimodule structure, which we transfer to $\Gamma(U, \ds_{Y\to X})$ via \eqref{SV isom}. Note that the left $D_B$-action on $\Gamma(U, \ds_{Y\to X})$ extends the usual left $B$-action, and the right $D_A$-action on $\Gamma(U, \ds_{Y\to X})$ is the one induced by the right $\pi^{-1}\ds_X$-action on $\ds_{Y\to X}$. We define the left $\ds_Y$-module structure on $\ds_{Y\to X}$ as the one obtained by gluing the $D_U$-module structure on $\Gamma(U, \ds_{Y\to X})$ over all affine opens $U \subset Y$. By Lemma \ref{transfer bimodule tool}, $\ds_{Y \to X}$ is a $(\ds_Y,\pi^{-1}\ds_X)$-bimodule.

Next, we consider transfer bimodules adapted for the $V$-filtration. It is immediate that
$$
\pi^*V^0\rs_X = \ds_{Y\to X} \otimes_{\bbc} V^0D_{\bba^r}
$$
is a $(V^0\rs_Y, \pi^{-1}V^0\rs_X)$-bimodule, and that
$$
\pi^*R^+_V\rs_X= \ds_{Y\to X}  \otimes_{\bbc} R^+_VD_{\bba^r}
$$
is an $(R^+_V\rs_Y, \pi^{-1}R^+_V\rs_X)$-bimodule. Since $V^0\rs_Y$ and $R^+_V\rs_Y$ are free as right $\ds_Y$-modules, we have
\begin{equation}\label{vr}
\pi^*V^0\rs_X= V^0\rs_Y \otimes_{\ds_Y}^L \ds_{Y\to X} 
\ \ \ \ 
\text{and}
\ \ \ \ 
\pi^*R^+_V\rs_X  = R^+_V\rs_Y \otimes_{\ds_Y}^L \ds_{Y\to X}.
\end{equation}
Thus we obtain
\begin{equation}\label{rees vr}
\pi^*R^+_V\rs_X = R^+_V\rs_Y \otimes^L_{V^0\rs_Y} \pi^*V^0\rs_X.
\end{equation}

\begin{definition}
Given a right $\ds_Y$-module $\ns$, we put
$$
\widetilde{\pi}_+(\ns) := \hs^0 R\pi_*(\ns \otimes^L_{\ds_Y} \ds_{Y \to X}),
$$
which is a right $\ds_X$-module.
\end{definition}

If $\ns$ is a right $V^0\rs_Y$-module, then $\widetilde{\pi}_+(\ns)$ is a right $V^0\rs_X$-module by \eqref{vr} because
\begin{equation}\label{pi tilde}
\widetilde{\pi}_+(\ns) = \hs^0 R\pi_*(\ns \otimes^L_{V^0\rs_Y} \pi^*V^0\rs_X).
\end{equation}

If $\ns= \bigoplus_{k\in \bbz} \ns_k T^k$ is a graded right $R^+_V\rs_Y$-module, then $\widetilde{\pi}_+(\ns)$ is a graded right $R^+_V\rs_X$-module. To see this, note that by Lemma \ref{derived direct image and direct sum}, we have
$$
\widetilde{\pi}_+(\ns) = \bigoplus_{k\in \bbz} \widetilde{\pi}_+(\ns_k)T^k,
$$
which defines the grading, and by \eqref{rees vr}, we have 
$$
\widetilde{\pi}_+(\ns) = \hs^0 R\pi_*(\ns \otimes^L_{R^+_V\rs_Y} \pi^*R^+_V\rs_X),
$$
which defines the $R^+_V\rs_X$-module structure. It is easy to see that the grading is compatible with the $R^+_V\rs_X$-module structure, so that $\widetilde{\pi}_+(\ns)$ is indeed a graded $R^+_V\rs_X$-module.

Now consider the morphism $\pi\times 1:Y \times \bba^r \to X \times \bba^r$. We claim that
$$
(\pi \times 1)^*V^0\ds_{X\times \bba^r} = \os_{Y \times \bba^r} \otimes_{(\pi\times 1)^{-1}\os_{X \times \bba^r}} (\pi\times 1)^{-1}V^0\ds_{X\times \bba^r}
$$
is a $(V^0\ds_{Y \times \bba^r}, (\pi\times 1)^{-1}V^0\ds_{X\times \bba^r})$-bimodule. The right $(\pi\times 1)^{-1}V^0\ds_{X\times \bba^r}$-module structure is simply given by multiplication on the right tensor factor. The left $V^0\ds_{Y \times \bba^r}$-module structure is given by the induced action of $V^0\ds_{Y \times \bba^r}=\ds_Y \boxtimes V^0\ds_{\bba^r}$ on $(\pi \times 1)^*V^0\ds_{X\times \bba^r}=\ds_{Y\to X}  \boxtimes V^0\ds_{\bba^r}$. To see why these two actions form a bimodule structure, by Lemma \ref{transfer bimodule tool}, it suffices to show that $\Gamma(U\times \bba^r, (\pi \times 1)^*V^0\ds_{X\times \bba^r})$ is a $(D_U\otimes_{\bbc}V^0D_{\bba^r}, D_X \otimes_{\bbc} V^0D_{\bba^r})$-bimodule for every affine open $U \subset Y$, which follows since
$$
\Gamma(U\times \bba^r, (\pi \times 1)^*V^0\ds_{X\times \bba^r}) = \Gamma(U,\ds_{Y \to X}) \otimes_{\bbc} V^0D_{\bba^r}.
$$

\begin{definition}
Given a right $V^0\ds_{Y \times \bba^r}$-module $\ns$, we define
$$
\widetilde{(\pi\times 1)}_{V,+}\ns := \hs^{0}R(\pi\times 1)_* (\ns \otimes_{V^0\ds_{Y\times \bba^r}}^L (\pi\times 1)^*V^0\ds_{X\times \bba^r} ),
$$
which is a right $V^0\ds_{X \times \bba^r}$-module.
\end{definition}

\begin{lemma}\label{box pull}
We have $L\pi^*V^0\rs_X= \pi^*V^0\rs_X$, $L\pi^*R^+_V\rs_X= \pi^*R^+_V\rs_X$, and 
$$
L(\pi \times 1)^*V^0\ds_{X \times \bba^r}= (\pi \times 1)^*V^0\ds_{X \times \bba^r}.
$$
\end{lemma}
\proof Since $\ds_X$ is free as a left $\os_X$-module, we have $L\pi^*\ds_X = \pi^*\ds_X$. It is then immediate that $L\pi^*V^0\rs_X= \pi^*V^0\rs_X$ and $L\pi^*R^+_V\rs_X= \pi^*R^+_V\rs_X$. For the final isomorphism, observe that
$$
L(\pi \times 1)^*V^0\ds_{X \times \bba^r}=L\pi^*\ds_X \boxtimes V^0\ds_{\bba^r} = \pi^*\ds_X \boxtimes V^0\ds_{\bba^r} = (\pi \times 1)^*V^0\ds_{X \times \bba^r}. \qed
$$

\begin{lemma}\label{push example}
We have $\widetilde{\pi}_+(V^0\rs_Y) = \pi_*\os_Y \otimes_{\os_X} V^0\rs_X$.
\end{lemma}
\proof Observe that
$$
\widetilde{\pi}_+(V^0\rs_Y) = \hs^0(R\pi_*\pi^*V^0\rs_X) =  \hs^0(R\pi_*L\pi^*V^0\rs_X) 
$$
$$
= \hs^0(R\pi_*\os_Y \otimes^L_{\os_X} V^0\rs_X)
= \hs^0(R\pi_*\os_Y ) \otimes_{\os_X} V^0\rs_X =\pi_*\os_Y \otimes_{\os_X} V^0\rs_X
$$
where the second, third, and fourth isomorphisms follow from Lemma \ref{box pull}, \cite[Lemma 36.22.1]{Stacks}, and the fact that $V^0\rs_X$ is free as a left $\os_X$-module, respectively. \qed

\section{Projective direct image}\label{projective direct image}

Throughout this section, we fix an integer $r\geq 1$ and a projective morphism $\pi: Y \to X$ such that $X$ is a smooth affine variety and the morphism $\rSpec_Y (\gr_{\bullet}^F\ds_Y) \to Y$ is of finite type. We denote the projections $Y \times \bba^r \to Y$ and $X \times \bba^r \to X$ by $p_Y$ and $p_X$, respectively.

\begin{lemma}\label{projective coherence}
Let $\ms$ be a coherent right $V^0\ds_{Y\times \bba^r}$-module. Then
$$
\hs^i R(\pi \times 1)_*(\ms \otimes^L_{V^0\ds_{Y\times \bba^r}} (\pi\times 1)^*V^0\ds_{X\times \bba^r})
$$
is a coherent right $V^0\ds_{X\times \bba^r}$-module for all $i\in \bbz$. In particular, $\widetilde{(\pi \times 1)}_{V,+}\ms$ is a coherent right $V^0\ds_{X\times \bba^r}$-module.
\end{lemma}
\proof Let $d=\dim Y \times \bba^r$. We follow the approach in \cite[Theorem 3.4.1]{Sab11}. Since $\pi \times 1$ is projective and $X \times \bba^r$ is affine, for every $\os_{Y \times \bba^r}$-coherent sheaf $\fs$, there exists an $\os_{Y \times \bba^r}$-finite locally free sheaf $\es$ that surjects onto $\gs$ (namely $\es=(\ls^{\otimes -N})^M$ for some $\pi$-ample line bundle $\ls$ and $N,M\geq 0$). So by Lemma \ref{surjection}, there exists an exact sequence
$$
\cdots \to\cs^{-2} \to \cs^{-1} \to \cs^0 \to \ms\to 0
$$
of right $V^0\ds_{Y \times \bba^r}$-modules such that $\cs^i = \es^i \otimes_{\os_{Y \times \bba^r}} V^0\ds_{Y \times \bba^r}$ for some $\os_{Y\times \bba^r}$-finite locally free sheaf $\es^i$ for each $i\leq 0$. We denote this resolution of $\ms$ by $\cs^{\bullet}$. We also consider the chain complex
$$
\bs := \cs^{\bullet} \otimes_{V^0\ds_{Y \times \bba^r}} (\pi \times 1)^*V^0\ds_{X \times \bba^r}.
$$

Since $\es^i$ is $\os_{Y \times \bba^r}$-flat, we have
$$
\cs^i \otimes^L_{V^0\ds_{Y\times \bba^r}} (\pi\times 1)^*V^0\ds_{X\times \bba^r}=\es^{i} \otimes_{\os_{Y \times \bba^r}} (\pi\times 1)^*V^0\ds_{X\times \bba^r},
$$
hence $\cs^i$ is $(-)\otimes^L_{V^0\ds_{Y\times \bba^r}} (\pi\times 1)^*V^0\ds_{X\times \bba^r}$-acyclic and $\bs = \ms \otimes^L_{V^0\ds_{Y\times \bba^r}} (\pi\times 1)^*V^0\ds_{X\times \bba^r}$. By Remark \ref{bounded cohomological dimension}, we then have
$$
\hs^i R(\pi \times 1)_*(\ms \otimes^L_{V^0\ds_{Y\times \bba^r}} (\pi\times 1)^*V^0\ds_{X\times \bba^r}) = \hs^i R(\pi \times 1)_*(\sigma_{\geq i-d-1}(\bs^{\bullet})).
$$
Thus it suffices to show that $R(\pi \times 1)_*(\sigma_{\geq i}(\bs^{\bullet}))$ has $V^0\ds_{X \times \bba^r}$-coherent cohomology sheaves for all $i\leq 0$. We proceed by descending induction on $i$. Before carrying out the inductive argument, we first show that $R(\pi \times 1)_*(\bs^i)$ has $V^0\ds_{X \times \bba^r}$-coherent cohomology sheaves for all $i\leq 0$. Note that $\bs^i = \es^i \otimes^L_{\os_{Y \times \bba^r}} L(\pi\times 1)^*V^0\ds_{X\times \bba^r}$ by Lemma \ref{box pull}. So we have
$$
R(\pi \times 1)_*( \bs^i ) 
= R(\pi \times 1)_*(\es^i ) \otimes_{\os_{X \times \bba^r}}^L V^0\ds_{X \times \bba^r}
$$
by \cite[Lemma 36.22.1]{Stacks}. Since $\pi\times 1$ is proper, $R(\pi \times 1)_*(\es^i )$ has $\os_{X \times \bba^r}$-coherent cohomology sheaves, hence $R(\pi \times 1)_*( \bs^i )$ has $V^0\ds_{X \times \bba^r}$-coherent cohomology sheaves.

We now return to the inductive argument. The base case follows since $R(\pi \times 1)_*(\sigma_{\geq 0}(\bs^{\bullet}))=R(\pi \times 1)_*( \bs^0)$. So let $i<0$ and assume that $R(\pi \times 1)_*(\sigma_{\geq i+1}(\bs^{\bullet}))$ has $V^0\ds_{X \times \bba^r}$-coherent cohomology sheaves. Applying $R(\pi \times 1)_*$ to the short exact sequence
$$
0 \to \sigma_{\geq i+1}(\bs^{\bullet}) \to \sigma_{\geq i}(\bs^{\bullet}) \to \bs^i[-i] \to 0
$$
yields the distinguished triangle
$$
R(\pi \times 1)_*(\sigma_{\geq i+1}(\bs^{\bullet})) \to R(\pi \times 1)_*(\sigma_{\geq i}(\bs^{\bullet}) ) \to R(\pi \times 1)_*(\bs^i)[-i].
$$
Since $R(\pi \times 1)_*(\sigma_{\geq i+1}(\bs^{\bullet}))$ and $R(\pi \times 1)_*(\bs^i)$ have $V^0\ds_{X \times \bba^r}$-coherent cohomology sheaves, by taking the long exact sequence in cohomology, we see that $R(\pi \times 1)_*(\sigma_{\geq i}(\bs^{\bullet}) )$ also has $V^0\ds_{X \times \bba^r}$-coherent cohomology sheaves, as desired. \qed

\begin{lemma}\label{push relation}
Given a coherent right $V^0\ds_{Y\times\bba^{r}}$-module $\ms$, we have the following natural isomorphism of right $V^0\rs_X$-modules:
$$
p_{X,*} (\widetilde{(\pi \times 1)}_{V,+}\ms ) \simeq \widetilde{\pi}(p_{Y,*}\ms).
$$
\end{lemma}
\proof We retain the notation from the proof of Lemma \ref{projective coherence}. In particular, $\cs^{\bullet} \to \ms$ denotes the resolution constructed there. We set $\qs = (\pi\times 1)^*V^0\ds_{X\times \bba^r}$, and note that
\begin{equation}\label{push transfer}
p_{Y,*}\qs=p_{Y,*}(\pi^*\ds_X \boxtimes V^0\ds_{\bba^r}) = \pi^*\ds_X \otimes_{\bbc} V^0D_{\bba^r}= \pi^*V^0\rs_X.
\end{equation}

We first claim that
\begin{equation}\label{first identity}
p_{Y,*}\ms \otimes^L_{V^0\rs_Y} p_{Y,*}\qs = p_{Y,*}\cs^{\bullet}\otimes_{V^0\rs_Y} p_{Y,*}\qs.
\end{equation}
To see this, observe that $p_{Y,*}\cs^{\bullet}\to p_{Y,*}\ms \to 0$ is exact since $\cs^{\bullet} \to \ms \to 0$ is an exact sequence of quasicoherent sheaves and $p_Y$ is an affine morphism. Also, note that
$$
p_{Y,*}\cs^i = p_{Y,*}(\es^i \otimes_{\os_{Y\times \bba^r}} V^0\ds_{Y \times \bba^r} )
= p_{Y,*}\es^i \otimes_{p_{Y,*}\os_{Y\times \bba^r}} V^0\rs_Y.
$$
Because $p_Y$ is an affine morphism and $\es_i$ is a flat quasicoherent $\os_{Y\times \bba^r}$-module, $p_{Y,*}\es^i$ is a flat $p_{Y,*}\os_{Y\times \bba^r}$-module, which implies that $p_{Y,*}\cs^i$ is $(-)\otimes_{V^0\rs_Y}p_{Y,*}(\qs)$-acyclic. Thus the claim follows.

We next claim that
\begin{equation}\label{second identity}
Rp_{Y,*}(\ms \otimes^L_{V^0\ds_{Y\times \bba^r}} \qs )
= p_{Y,*}\cs^{\bullet}\otimes_{V^0\rs_Y} p_{Y,*}\qs.
\end{equation}
To see this, recall from the proof of Lemma \ref{projective coherence} that $\ms \otimes^L_{V^0\ds_{Y\times \bba^r}} \qs  = \cs^{\bullet}\otimes_{V^0\ds_{Y\times \bba^r}} \qs$. Since $\qs$ is quasicoherent, $\cs^i\otimes_{V^0\ds_{Y\times \bba^r}} \qs  = \es^i \otimes_{\os_{Y\times \bba^r}} \qs $ is also quasicoherent. In particular, $\cs^i\otimes_{V^0\ds_{Y\times \bba^r}} \qs$ is $p_{Y,*}$-acyclic, hence $Rp_{Y,*}(\ms \otimes^L_{V^0\ds_{Y\times \bba^r}} \qs )
=p_{Y,*}(\cs^{\bullet}\otimes_{V^0\ds_{Y\times \bba^r}} \qs )$.
Because $p_{Y,*}\cs^i= p_{Y,*}\es^i \otimes_{p_{Y,*}\os_{Y\times \bba^r}} V^0\rs_Y$, we have
$$
p_{Y,*}(\cs^i\otimes_{V^0\ds_{Y\times \bba^r}} \qs ) = p_{Y,*}\es^i \otimes_{p_{Y,*}\os_{Y\times \bba^r}} p_{Y,*}\qs 
=p_{Y,*}\cs^i \otimes_{V^0\rs_Y} p_{Y,*}\qs,
$$
which implies the claim.

We now define the following isomorphism of right $\pi^{-1}\rs_{X}$-modules obtained by composing the isomorphisms in \eqref{first identity} and \eqref{second identity}:
$$
\phi_{\ms}=\phi_{\ms}^{\cs^{\bullet}}:p_{Y,*}\ms \otimes^L_{V^0\rs_Y} p_{Y,*}\qs \xrightarrow{\sim} Rp_{Y,*}(\ms \otimes^L_{V^0\ds_{Y\times \bba^r}} \qs ).
$$ 
We claim that $\phi_{\ms}$ is indeed independent of the choice of $\cs^{\bullet}$ and functorial in $\ms$. To see this, it suffices to show that if $\psi:\ms\to\ms'$ is a morphism of coherent right $V^{0}\ds_{Y\times\bba^{r}}$-modules and $\cs'^{\bullet}$ is a resolution of $\ms'$ as constructed in the proof of Lemma \ref{projective coherence}, then we have $\phi_{\ms'}^{\cs'^{\bullet}}\circ (p_{Y,*}(\psi)\otimes  1)=Rp_{Y,*}(\psi\otimes 1)\circ\phi_{\ms}^{\cs^{\bullet}}$. This is straightforward once we represent the morphism $\cs^{\bullet} \to \cs'^{\bullet}$ in $D(V^{0}\ds_{Y\times\bba^{r}}^{op})$ corresponding to $\psi$ as a roof $\cs^{\bullet} \from \ks^{\bullet} \to \cs'^{\bullet}$, where $\ks^{\bullet}$ is a bounded above complex of $V^{0}\ds_{Y\times\bba^{r}}$-modules such that $\ks^{i}=\vs^{i}\otimes_{\os_{Y\times\bba^{r}}}V^{0}\ds_{Y\times\bba^{r}}$ for some $\os_{Y\times\bba^{r}}$-finite locally free sheaf $\vs^{i}$ for each $i$, and $\ks^{\bullet} \to \cs'^{\bullet}$ is a quasi-isomorphism (such a roof exists by \cite[Lemma 13.15.4]{Stacks}).

Since $R(\pi \times 1)_*(\ms \otimes^L_{V^0\ds_{Y\times \bba^r}} (\pi\times 1)^*V^0\ds_{X\times \bba^r})$ has quasicoherent cohomology sheaves by Lemma \ref{projective coherence} and $p_X$ is an affine morphism, we have
$$
p_{X,*} (\widetilde{(\pi \times 1)}_{V,+}\ms )
= \hs^0 Rp_{X,*}R(\pi \times 1)_*(\ms \otimes^L_{V^0\ds_{Y\times \bba^r}} \qs)
=\hs^0 R\pi_*Rp_{Y,*}(\ms \otimes^L_{V^0\ds_{Y\times \bba^r}} \qs).
$$
Applying $\hs^0 R\pi_*$ to $\phi_{\ms}$ then yields the desired isomorphism
$$
p_{X,*} (\widetilde{(\pi \times 1)}_{V,+}\ms ) \simeq \hs^0 R\pi_*(p_{Y,*}\ms \otimes^L_{V^0\rs_{Y}} p_{Y,*}\qs)
=\widetilde{\pi}(p_{Y,*}\ms),
$$
where the last isomorphism follows from \eqref{pi tilde}. \qed

\begin{lemma}\label{projective graded coherence}
If $\ns$ is a coherent graded right $R^+_V\rs_Y$-module, then $\widetilde{\pi}_+(\ns)$ is a coherent graded right $R^+_V\rs_X$-module.
\end{lemma}
\proof The argument is identical to that of Lemma \ref{projective coherence} after replacing $Y\times \bba^r$, $X\times \bba^r$, $\pi\times 1$, $\os_{Y\times \bba^r}$, $\os_{X\times \bba^r}$, $V^0\ds_{Y \times \bba^r}$, $V^0\ds_{X \times \bba^r}$, $\os_{Y \times\bba^{r}}$-coherent, $\os_{Y \times\bba^{r}}$-finite locally free, and Lemma \ref{surjection} with $Y$, $X$, $\pi$, $\os_Y$, $\os_X$, $R^+_V\rs_Y$, $R^+_V\rs_X$, graded $\os_{Y}$-coherent, graded $\os_{Y}$-finite locally free, and Lemma \ref{graded surjection}, respectively. We thus omit the proof. \qed

\begin{corollary}\label{projective good}
Let $U^{\bullet}\ns$ be a good filtration of a right $V^0\rs_Y$-module $\ns$. If we put 
$$
\im(\widetilde{\pi}_+(U^k\ns)) := \im(\widetilde{\pi}_+(U^k\ns) \to  \widetilde{\pi}_+(\ns) ),
$$
then $(\im(\widetilde{\pi}_+(U^k\ns)) )_{k\geq 0}$ is a good filtration of $\widetilde{\pi}_+(\ns)$.
\end{corollary}
\proof Consider the inclusion $\bigoplus_{k\geq 0} U^k\ns T^k \subset \bigoplus_{k\geq 0} \ns T^k$ of graded right $R^+_V\rs_Y$-modules. Applying $\widetilde{\pi}_+$ yields the following morphism of graded $R^+_V\rs_X$-modules:
$$
\phi: \bigoplus_{k\geq 0} \widetilde{\pi}_+(U^k\ns) T^k \to \bigoplus_{k\geq 0} \widetilde{\pi}_+(\ns) T^k.
$$
By Lemma \ref{projective graded coherence}, $\bigoplus_{k\geq 0} \widetilde{\pi}_+(U^k\ns) T^k$ is a coherent $R^+_V\rs_X$-module. Note that $\bigoplus_{k\geq 0} \widetilde{\pi}_+(\ns) T^k$ is quasicoherent since $\widetilde{\pi}_+(\ns)=\widetilde{\pi}_+(U^0\ns)$ is a coherent $V^0\rs_X$-module by Lemma \ref{coherent pieces}. Thus $\im \phi=\bigoplus_{k\geq 0} \im(\widetilde{\pi}_+(U^k\ns)) T^k$ is a coherent graded right $R^+_V\rs_X$-module by Remark \ref{coherent char}, as desired. \qed

\begin{remark}\label{proper smooth}
All the results in this section continue to hold under the weaker assumption that $\pi$ is proper, provided that $Y$ is smooth. Indeed, under these modified assumptions, while the $\os_{Y\times \bba^r}$-coherent sheaves $\es^i$ in the proofs of Lemmas \ref{projective coherence} and \ref{push relation} can no longer be assumed to be locally free, the proofs still go through because $\ds_{Y \times \bba^r}$ and $\ds_Y$ are now locally free as left $\os_Y$-modules.
\end{remark}

\section{Proof of the main theorem}\label{proof of the main theorem}

We begin this section by proving Theorem \ref{main result}. We then deduce Theorem \ref{main result for hypersurfaces} as a consequence.

\noindent{\it Proof of Theorem \ref{main result}}. Let $p_Y$ and $p_X$ denote the projections $Y \times \bba^r \to Y$ and $X\times \bba^r \to X$, respectively, and put $\rs_Y=p_{Y,*}\ds_{Y\times \bba^r}$ and $\rs_X=p_{X,*}\ds_{X\times \bba^r}$. Also, let $p$ and $q$ denote the morphisms $U/G \onto Y$ and $U\to U/G$, respectively. Recall that $Z\subset X$ is the closed subscheme defined by $f_1,\dots,f_r$. Since $Y$ is normal, we can consider the right $\rs_Y$-module $B_{\pi^*\bff}^{\omega}$. By Lemma \ref{quot sing finite type}, we can use the results in Sections \ref{coherent modules and good filtrations} and \ref{projective direct image}.

Recall that $u= \sum_{\beta} u_{\beta}\partial_t^{\beta}\delta_{\bff} \in \Gamma(X,B_{\bff})$. We consider the section
$$
v := \sum_{\beta} (-1)^{|\beta|}(u_{\beta}dx_1\wedge \cdots \wedge dx_n) \partial_t^{\beta}\delta^{\omega}_{\bff} \in \Gamma(X,B_{\bff}^{\omega})
$$
and set $v^* := \pi^*v \in \Gamma(Y,B_{\pi^*\bff}^{\omega})$ (recall Definition \ref{bfomega pullback}). Note that $b_u(s)=b_v(-s-r)$ by Lemma \ref{left right switch}. Since
$$
((\pi \circ p \circ q)^*v )|_{U_i} = \sum_{\beta} (-1)^{|\beta|}(\pi^*(u_{\beta}) g_idy_1\wedge \cdots \wedge dy_n) \partial_t^{\beta}\delta^{\omega}_{\pi_i^*\bff} \in \Gamma(U_i,B_{\pi_i^*\bff}^{\omega}),
$$
we also have $b_{\pi_i^*u}(s) = b_{((\pi \circ p \circ q)^*v )|_{U_i}}(-s-r)$ by Lemma \ref{left right switch}. It then follows from Remark \ref{b function local} that $b(s)=b_{(\pi \circ p \circ q)^*v}(-s-r)$. Thus, the theorem statement is equivalent to showing that
$$
b_v(s) \ | \ b_{(\pi \circ p \circ q)^*v}(s)b_{(\pi \circ p \circ q)^*v}(s-1) \cdots b_{(\pi \circ p \circ q)^*v}(s-N)
$$
for some $N \geq 0$. By Corollary \ref{finite group b function} and Theorem \ref{etale pull b}, we have 
$$
b_{v^*}(s) = b_{p^*v^*}(s) \ | \ b_{q^*(p^*v^*)}(s),
$$
and by \eqref{composition}, we have $q^*(p^*v^*) = (\pi \circ p \circ q)^*v$, hence $b_{v^*}(s) \ | \ b_{(\pi \circ p \circ q)^*v}(s)$. It therefore suffices to show that
$$
b_v(s) \ | \ b_{v^*}(s)b_{v^*}(s-1) \cdots b_{v^*}(s-N)
$$
for some $N \geq 0$, which the rest of this proof is devoted to showing.

We consider the coherent right $V^0\rs_Y$-module $\ns := v^* \cdot V^0\rs_Y \subset B_{\pi^*\bff}^{\omega}$ and the morphism $\phi: V^0\rs_Y \onto \ns,P \mapsto v^{*}\cdot P$. Note that $\widetilde{\pi}_+(V^0\rs_Y)= V^0\rs_X$ by Lemma \ref{push example} and Zariski's Main Theorem. We denote by 
$$
\widetilde{v}:=\widetilde{\pi}_+(\phi)(1)  \in \Gamma(X,\widetilde{\pi}_+(\ns) )
$$
the image of $1\in V^0\rs_X$ under the morphism $\widetilde{\pi}_+(\phi):V^0\rs_X=\widetilde{\pi}_+(V^0\rs_Y) \to \widetilde{\pi}_+(\ns)$.

We claim that there exists some $M \geq 0$ such that
\begin{equation}\label{squish}
\widetilde{\pi}_+(\ns) \cdot V^M\rs_X \subset \widetilde{v} \cdot V^0\rs_X.
\end{equation}
To see this, let $v^{**} \in \Gamma(X\times \bba^r,\hs^r_{\Gamma_{\pi^*\bff}}(\omega_{Y\times \bba^r}))$ denote the section corresponding to $v^*$ via
$$
\Gamma(X\times \bba^r,\hs^r_{\Gamma_{\pi^*\bff}}(\omega_{Y\times \bba^r})) = \Gamma(X, B_{\pi^*\bff}^{\omega}).
$$
We also consider the surjection $\phi': V^0\ds_{Y\times \bba^r} \onto v^{**}\cdot V^0\ds_{Y\times \bba^r},P \mapsto v^{**}\cdot P$ and set
$$
\qs = \coker (\widetilde{(\pi\times 1)}_{V,+}(\phi'):  \widetilde{(\pi\times 1)}_{V,+}(V^0\ds_{Y\times \bba^r})
\to
\widetilde{(\pi\times 1)}_{V,+}(v^{**}\cdot V^0\ds_{Y\times \bba^r})).
$$
Note that $\widetilde{(\pi\times 1)}_{V,+}(v^{**} \cdot V^0\ds_{Y \times \bba^r})$ is supported on $\Gamma_{\bff}\subset X \times \bba^r$ since $v^{**} \cdot V^0\ds_{Y \times \bba^r}$ is supported on $\Gamma_{\pi^*\bff}=(\pi\times 1)^{-1}(\Gamma_{\bff}) \subset Y\times \bba^r$. Since $\pi\times 1$ is an isomorphism over $(X \setminus Z)\times \bba^r$, it follows that $\widetilde{(\pi\times 1)}_{V,+}(\phi')|_{(X \setminus Z) \times \bba^r}$ is surjective, hence $\qs$ is supported on $Z\times \bba^r \subset X\times \bba^r$. Thus $\qs$ is supported on $\Gamma_{\bff} \cap (Z\times \bba^r) \subset X\times 0$. By Lemmas \ref{projective coherence} and \ref{torsion}, there exists some $M \geq 0$ such that $\qs \cdot V^M\ds_{X\times \bba^r} = 0$, which implies that $p_{X,*}(\qs) \cdot V^M\rs_X =0$. The claim then follows since $p_{X,*}(\qs) = \coker (\widetilde{\pi}_+(\phi))= \widetilde{\pi}_+(\ns)/ \widetilde{v}\cdot V^0\rs_X$ by Lemma \ref{push relation}.

Consider the good filtration $(v^* \cdot V^k\rs_Y)_{k\geq 0}$ of $\ns$. By Lemma \ref{projective good}, if we put
$$
\im(\widetilde{\pi}_+(v^* \cdot V^k\rs_Y)) := \im( \widetilde{\pi}_+(v^*  \cdot V^k\rs_Y) \to \widetilde{\pi}_+(\ns)  ),
$$
then $(\im(\widetilde{\pi}_+(\ns \cdot V^k\rs_Y)))_{k\geq 0}$ is a good filtration of $\widetilde{\pi}_+(\ns)$. By Lemma \ref{artin rees}, there exists some $N \geq 0$ such that 
\begin{equation}\label{artin rees consequence}
\im(\widetilde{\pi}_+(v^* \cdot V^{N+1}\rs_Y)) \subset \widetilde{\pi}_+(\ns) \cdot V^{M+1}\rs_X.
\end{equation}
Now set
$$
a(s):=b_{v^*}(s)b_{v^*}(s-1)\cdots b_{v^*}(s-N).
$$
By Lemma \ref{right moving}, the morphism $a(s): \ns \to \ns, w \mapsto w \cdot a(s)$, which is $\ds_Y$-linear, factors through $v^*\cdot V^{N+1}\rs_Y \subset \ns$. Thus we can consider the following commutative diagram of right $\ds_Y$-modules:
\[
\begin{tikzcd}
\ns \arrow[r] \arrow[rr, bend right,"a(s)"] & v^*\cdot V^{N+1}\rs_Y \arrow[r, hook] & \ns.
\end{tikzcd}
\]
Applying $\widetilde{\pi}_+$ yields the commutative diagram
\[
\begin{tikzcd}
\widetilde{\pi}_+(\ns) \arrow[r] \arrow[rr, bend right,"a(s)"] & \widetilde{\pi}_+(v^*\cdot V^{N+1}\rs_Y) \arrow[r] & \widetilde{\pi}_+(\ns),
\end{tikzcd}
\]
whose commutativity implies that $\widetilde{\pi}_+(\ns) \cdot a(s) \subset \im(\widetilde{\pi}_+(v^*\cdot V^{N+1}\rs_Y))$. Thus we have
$$
\widetilde{\pi}_+(\ns) \cdot a(s) 
 \subset \im(\widetilde{\pi}_+(v^* \cdot V^{N+1}\rs_Y)) 
\subset \widetilde{\pi}_+(\ns) \cdot V^{M+1}\rs_X
\subset \widetilde{v} \cdot V^1\rs_X,
$$
where the second and third inclusions follow from \eqref{artin rees consequence} and \eqref{squish}, respectively. Since $\widetilde{v}$ is a section of $\widetilde{\pi}_+(\ns)$, we then obtain
\begin{equation}\label{pushed b functional equation}
\widetilde{v}  \cdot a(s) \in\widetilde{v} \cdot V^1\rs_X.
\end{equation}

We next show that there is a morphism of right $V^0\rs_X$-modules $\psi: \widetilde{v} \cdot V^0\rs_X  \to v \cdot V^0\rs_X$ such that $\psi(\widetilde{v})=v$. Let $j: X \setminus Z  \to X$ denote the inclusion. Note that the natural morphism $v\cdot V^0\rs_X \to j_*j^*(v\cdot V^0\rs_X)$ is injective since $v\cdot V^0\rs_X \subset B_{\bff}^{\omega}$ and $B_{\bff}^{\omega}$ is $\os_X$-torsion-free by Lemma \ref{bfomega action}. Thus it suffices to show that there is a morphism of right $V^0\rs_X$-modules $\widetilde{v} \cdot V^0\rs_X \to j_*j^*(v\cdot V^0\rs_X)$ that maps $\widetilde{v}$ to $v|_{X \setminus Z}$. Since $\pi$ is an isomorphism over $X \setminus Z$, it is straightfoward to see that $j^*(\widetilde{v} \cdot V^0\rs_X)=j^*(v\cdot V^0\rs_X)$ and $\widetilde{v}|_{X \setminus Z} = v|_{X \setminus Z}$. The desired morphism is then the composition $\widetilde{v} \cdot V^0\rs_X \to j_*j^*(\widetilde{v} \cdot V^0\rs_X) =  j_*j^*(v\cdot V^0\rs_X)$. 

Finally, since $\widetilde{v}\cdot a(s) = \widetilde{v} \cdot P$ for some $P\in V^1\rs_X$ by \eqref{pushed b functional equation}, it follows that 
$$
v\cdot a(s) = \psi(\widetilde{v}\cdot a(s) ) = \psi(\widetilde{v} \cdot P)= v \cdot P
$$
by $V^0\rs_X$-linearity of $\psi$. Hence $b_v(s) \ | \ a(s)$, as desired. \qed

\begin{remark}
By Remark \ref{proper smooth}, Theorem \ref{main result} continues to hold under the weaker assumption that $\pi$ is proper, provided that $Y$ is smooth.
\end{remark}

Before we prove Theorem \ref{main result for hypersurfaces}, we begin with the following remark.

\begin{remark}\label{the v filtration on bf}
Let $f\in \os_X(X)$ be a regular function on a smooth variety $X$. The $V$-filtration $(V^{\lambda}B_f)_{\lambda}$ on $B_f$, originally constructed by Malgrange in \cite{Mal83} as a filtration indexed by integers, is a certain decreasing filtration indexed by rational numbers that is uniquely characterized by a few properties (for details, we refer to \cite{Kas83}, \cite[Section 3.1]{Sai88}, and \cite[Section 1]{BMS06}). Recall from Remark \ref{rational roots} that for every $u\in \Gamma(X,B_f)$, the $b$-function $b_u(s)$ is a nonzero polynomial such that all its roots are rational. The $V$-filtration has the following description in terms of $b$-functions, due to Sabbah \cite{Sab87}:
$$
\Gamma(X,V^{\lambda}B_f) = \{u\in \Gamma(X,B_f) \ | \ \text{all roots of $b_u(s)$ are $\leq -\lambda$}\}.
$$
\end{remark}

\noindent{\it Proof of Theorem \ref{main result for hypersurfaces}}. With the notation introduced before Theorem \ref{main result}, note that 
$$
\pi_i^*(\partial_t^m\delta_f) = v y_{1}^{k_{i1}}\cdots y_{n}^{k_in}\partial_t^m\delta_{u y_{1}^{a_{i1}} \cdots y_{n}^{a_{in}}}
$$
for every $m\geq 0$ and $i\in I$. The first part of the theorem concerning the form of the roots of $b_f(s)$ follows immediately from Theorem \ref{main result} since $b_{\pi^*_i(\delta_f)}(s)$ divides $\prod_{j=1}^n \prod_{\ell=1}^{a_{ij}} (s+\frac{k_{ij}+\ell}{a_{ij}})$ by \cite[Lemmas 2.6 and 2.7]{DM22}. 

For the final statement describing a lower bound for the minimal exponent when $a_{i1}\in \{0,1\}$ and $k_{i1}=0$ for all $i\in I$, let us write
$$
\min_{i\in I,j\geq 2,a_{ij}\neq 0}\frac{k_{ij}+1}{a_{ij}}= q+\gamma
$$
for $q\in \bbz_{\geq 0}$ and $\gamma\in (0,1]$. By a result of Saito (see \cite[(1.3.8)]{Sai17}), $\widetilde{\alpha}(f)\geq q+\gamma$ if and only if $\partial_t^q\delta_f \in \Gamma(X, V^{\gamma}B_f)$. By Remark \ref{the v filtration on bf}, this is equivalent to the fact that all roots of $b_{\partial^q_t\delta_f}(s)$ are $\leq -\gamma$. By Theorem \ref{main result}, we have
$$
\max \{\lambda\in \bbq \ | \ b_{\partial^q_t\delta_f}(\lambda) = 0 \} \leq \max \{\lambda\in \bbq \ | \ b_{\pi_i^*(\partial_t^q\delta_f)}(\lambda) = 0 \text{ for some $i\in I$} \}.
$$
By \cite[Lemmas 2.6 and 2.7]{DM22}, $b_{\pi_i^*(\partial_t^q\delta_f)}(s)$ divides $(s+1)\prod_{j=2}^n \prod_{\ell=1}^{a_{ij}} (s-q+\frac{k_{ij}+\ell}{a_{ij}})$, hence
$$
\max \{\lambda\in \bbq \ | \ b_{\partial^q_t\delta_f}(\lambda) = 0 \}  \leq \max\{-1, \max_{i\in I,j\geq 2,a_{ij}\neq 0}\{q - \frac{k_{ij}+1}{a_{ij}}\}\}
=\max\{-1,-\gamma\}=-\gamma,
$$
as desired. \qed

\section{Estimates of zero loci of Bernstein-Sato ideals}\label{Estimates of zero loci of Bernstein-Sato ideals}

The main goal of this section is to prove Theorem \ref{main result for bernstein-sato ideals}. The proof closely follows that of Theorem \ref{main result for hypersurfaces}, so we begin by establishing the analogous definitions and results.

\subsection{The $W$-filtration}\label{}

Throughout this section, we fix a variety $X$ and an integer $r\geq 1$. In this section, we define and prove some basic properties of the $W$-filtration on $\ds_{X\times \bba^r}$. 

We denote the standard coordinates on $\bba^r$ by $t_1,\dots,t_r$ and the projection $X \times \bba^r \to X$ by $p_X$, and put
$\rs_X := p_{X,*}\ds_{X\times \bba^r} = \ds_X\otimes_{\bbc} D_{\bba^r}$
For each $i\in \{1,\dots,r\}$, we denote by $\is_{i} \subset \os_{X\times \bba^r}$ the coherent sheaf of ideals generated by $t_i$ and set
$$
V_{i}^k\ds_{X\times \bba^r} := \{P \in \ds_{X\times \bba^r} \ | \ P ( \is_{i}^j) \subset \is_{i}^{j+k} \text{ for all $j\in \bbz$}\},
$$
where $\is^j_{i}=\os_{X\times \bba^r}$ for $j\leq 0$. 
The $W$-filtration on $\ds_{X\times \bba^r}$ is the decreasing filtration $(W^k\ds_{X\times \bba^r})_{k\in \bbz}$ given by
$$
W^k\ds_{X\times \bba^r} := \bigcap_{i=1}^{r}V^{k}_i\ds_{X\times \bba^r}.
$$
This filtration also appears in \cite[Definition 2.1.2]{Wu17}, where it is defined differently. We then consider the $W$-filtration $(W^k\rs_X)_{k\in \bbz}$ on $\rs_X$ given by
$$
W^k\rs_X :=p_{X,*} W^k\ds_{X\times \bba^r},
$$
as well as the corresponding $\bbz_{\geq 0}$-graded Rees algebra
$$
R^+_W\rs_X := \bigoplus_{k\geq 0} W^k\rs_X T^k \subset \rs_X[T].
$$
By Lemma \ref{qcoh V filtration}, $W^k\ds_{X \times \bba^r}$ is a quasicoherent $\os_{X \times \bba^r}$-bimodule for each $k$, hence $W^0\ds_{X\times \bba^r}$, $W^0\rs_X$, and $R^+_W\rs_X$ are quasicoherent sheaves of rings.

\begin{lemma}\label{BS explicit V filtration}
For $k\in \bbz$, we have
$$
W^k\rs_X =\bigoplus_{\alpha_{i}-\beta_{i}\geq k} \ds_X t^{\alpha}\partial_t^{\beta}=\bigoplus_{\alpha_{i}-\beta_{i}\geq k} \ds_X \partial_t^{\beta} t^{\alpha},
$$
where both direct sums run over all $\alpha,\beta\in \bbz^r_{\geq 0}$ such that $\alpha_{i}-\beta_{i}\geq k$ for each $i\in \{1,\dots,r\}$.
\end{lemma}
\proof This follows immediately from Lemma \ref{explicit V filtration}.\qed

By Lemma \ref{BS explicit V filtration}, we have
\begin{equation}\label{BS push generation from zero}
W^0\rs_X=\ds_{X}\langle s_{1},\dots,s_{r},t_{1},\dots,t_{r}\rangle
\ \ \ \ 
\text{and}
\ \ \ \ 
W^k\rs_X = (t_1\cdots t_r)^k W^0\rs_X
\end{equation}
for $k\geq 0$. The following lemma is an immediate consequence of quasicoherence and \eqref{BS push generation from zero}.

\begin{lemma}\label{BS generation from zero}
We have $W^k\ds_{X\times \bba^r} = (t_1\cdots t_r)^k\cdot W^0\ds_{X\times \bba^r}$ for $k\geq 0$.
\end{lemma}

We put
$$
W^kD_{\bba^r} := \bigoplus_{|\alpha|-|\beta|\geq k} \bbc t^{\alpha} \partial_t^{\beta} \subset D_{\bba^r}
\ \ \ \ 
\text{and}
\ \ \ \ 
R^+_WD_{\bba^r}:=\bigoplus_{k\geq 0}W^kD_{\bba^r}T^k \subset D_{\bba^r}[T].
$$
By Lemma \ref{BS explicit V filtration}, we have $W^k\rs_X = \ds_X\otimes_{\bbc} W^kD_{\bba^r}$ and $R^+_W\rs_X= \ds_X \otimes_{\bbc} R^+_WD_{\bba^r}$.

\subsection{Bernstein-Sato ideals}\label{}

Throughout this section, we fix a variety $X$ and an integer $r\geq 1$. In this section we reformulate the definition of the relative Bernstein-Sato ideal given in the Introduction in terms of the $W$-filtration and prove some of its basic properties. We denote the standard coordinates on $\bba^r$ by $t_1,\dots,t_r$ and consider the operators $s_{i} = - \partial_{t_i} t_i \in W^0D_{\bba^r}$ for $1\leq i\leq r$. We put $\rs_X = \ds_X\otimes_{\bbc} D_{\bba^r}$.

Given a left $\rs_X$-module $\ms$ and a section $u\in \Gamma(X,\ms)$, the Bernstein-Sato ideal $\BS_{u}$ of $u$ is defined as
$$
\BS_{u}=\{ b\in \bbc[s_{1},\dots,s_{r}]\ |\ b\cdot u\in W^{0}\rs_{X} \cdot u\}.
$$
Similarly, given a right $\rs_X$-module $\ns$ and a section $v\in \Gamma(X,\ns)$, the Bernstein-Sato ideal $\BS_{v}$ of $v$ is defined as 
$$
\BS_{v}=\{ b\in \bbc[s_{1},\dots,s_{r}]\ |\ v\cdot b\in v\cdot W^{0}\rs_{X} \}.
$$

\begin{remark}\label{relative coincidence}
Suppose that $X$ is smooth. Given regular functions $g,f_{1},\dots,f_{r}\in \os_{X}(X)$, let $\bff$ denote the $r$-tuple $(f_{1},\dots,f_{r})$. The Bernstein-Sato ideal $\BS_{\bff,g}$ of $\bff$ relative to $g$, which was defined in the Introduction, admits the following equivalent definition by \eqref{specialization}:
$$
\BS_{\bff,g}:=\{ b\in \bbc[s_{1},\dots,s_{r}]\ |\ b\cdot g\delta_{\bff}\in \ds_{X}[s_{1},\dots,s_{r}]  \cdot g\prod_{i=1}^{r}f_{i}\delta_{\bff}\}.
$$
Since $W^{0}\rs_{X}\cdot g\prod_{i=1}^{r}f_{i}\delta_{\bff}=\ds_{X}[s_{1},\dots,s_{r}]  \cdot g\prod_{i=1}^{r}f_{i}\delta_{\bff}$ by \eqref{BS push generation from zero}, we have $\BS_{\bff,g}=\BS_{g\delta_{\bff}}$.
\end{remark}

\begin{remark}\label{BS b function local}
The Bernstein-Sato ideal is local in the following sense. Given an open cover $X=\bigcup_i U_i$, a right $\rs_X$-module $\ns$, and a section $v\in \Gamma(X,\ns)$, we have $\BS_{v} = \bigcap_i \BS_{v|_{U_i}}$. 
\end{remark}

\begin{lemma}\label{BS left moving}
Suppose that $X$ is smooth. Given regular functions $g,f_{1},\dots,f_{r}\in \os_{X}(X)$, if we denote the $r$-tuple $(f_{1},\dots,f_{r})$ by $\bff$, then for every element $b(s_{1},\dots,s_{r})\in \BS_{\bff,g}$, we have
$$
b(s_{1}+k,\dots,s_{r}+k)(W^k\rs_X\cdot g\delta_{\bff})  \subset W^{k+1}\rs_X\cdot g\delta_{\bff} 
$$
for every $k\geq 0$.
\end{lemma}
\proof This follows from \eqref{bf action} and \eqref{BS push generation from zero}, which imply that
$$
W^k\rs_X\cdot g\delta_{\bff}=(t_1\cdots t_r)^k\ds_{X}[s_{1},\dots,s_{r}] \cdot g\delta_{\bff}.\qed
$$

\begin{lemma}\label{BS right moving}
Suppose that $X$ is normal. Given regular functions $f_{1},\dots,f_{r}\in \os_{X}(X)$ and a section $v\in \Gamma(X,\omega_{X})$, if we denote the $r$-tuple $(f_{1},\dots,f_{r})$ by $\bff$, then for every element $b(s_{1},\dots,s_{r})\in \BS_{v\delta_{\bff}^{\omega}}$, we have
$$
(v\delta^{\omega}_{\bff}\cdot W^k\rs_X) b(s_{1}-k,\dots,s_{r}-k) \subset v\delta^{\omega}_{\bff}\cdot W^{k+1}\rs_X
$$
for every $k\geq 0$. 
\end{lemma}
\proof This follows from Lemma \ref{bfomega action} and \eqref{BS push generation from zero}, which imply that
$$
v\delta^{\omega}_{\bff}\cdot W^k\rs_X=v\delta^{\omega}_{\bff}\cdot (t_1\cdots t_r)^k\ds_{X}[s_{1},\dots,s_{r}].\qed
$$

\subsection{Properties of Bernstein-Sato ideals}\label{}

The proofs of the following three results are essentially the same as those of Lemma \ref{direct summand b function}, Lemma \ref{left right switch}, and Theorem \ref{etale pull b}, respectively, so we omit them.

\begin{lemma}\label{BS direct summand b function}
Fix an integer $r\geq 1$ and denote the standard coordinates on $\bba^r$ by $t_1,\dots,t_r$. Given a direct summand $R \subset S$ of finite type $\bbc$-algebras with splitting $\rho$, let $\rho'$ denote the map $\rho \otimes 1 : S \otimes_{\bbc}\bbc[t_1,\dots,t_r] \to R \otimes_{\bbc}\bbc[t_1,\dots,t_r]$ and let $\ms$ and $\ns$ be quasicoherent right $\ds_{\Spec R} \otimes_{\bbc} D_{\bba^r}$- and right $\ds_{\Spec S} \otimes_{\bbc} D_{\bba^r}$-modules, respectively, such that $\Gamma(\Spec R,\ms)$ is an $R[t_1,\dots,t_r]$-submodule of $\Gamma(\Spec S,\ns)$. Suppose there exists an $R[t_1,\dots,t_r]$-module map $\tau:\Gamma(\Spec S,\ns)\to \Gamma(\Spec R,\ms)$ such that $(\Gamma(\Spec R,\ms) \subset \Gamma(\Spec S,\ns),\tau)$ is a differential direct summand compatible with $\rho'$. Then for every $v\in \Gamma(\Spec R,\ms)$, if we let $v^*$ denote the image of $v$ in $\Gamma(\Spec S,\ns)$, we have $\BS_v\supset \BS_{v^*}$.
\end{lemma}

\begin{lemma}\label{BS left right switch}
Let $X$ be a smooth variety with global coordinates $x_1,\dots,x_n\in \os_X(X)$, let $g,f_1,\dots,f_r\in \os_X(X)$ be regular functions, let $\bff$ denote the $r$-tuple $(f_1,\dots,f_r)$, and consider the section $v=(g dx_1\wedge \cdots \wedge dx_n) \delta^{\omega}_{\bff}\in \Gamma(X, B^{\omega}_{\bff})$. Then for $b(s_{1},\dots,s_{r})\in  \bbc[s_{1},\dots,s_{r}]$, we have $b(s_{1},\dots,s_{r})\in \BS_{v}$ if and only if $b(-s_{1}-1,\dots,-s_{r}-1)\in \BS_{g\delta_{\bff}}$. 
\end{lemma}

\begin{theorem}\label{BS etale pull b}
Let $\pi: X \to Y$ be a surjective \'{e}tale morphism of varieties of the same pure dimension such that $Y$ is normal, let $f_1,\dots,f_r\in \os_Y(Y)$ be regular functions, and let $\bff$ denote the $r$-tuple $(f_1,\dots,f_r)$. Given a section $v\in \Gamma(Y, B_{\bff}^{\omega})$, we have $\BS_v = \BS_{\pi^*v}$.
\end{theorem}

The argument for the following corollary is identical to that of Corollary \ref{finite group b function} after replacing Lemma \ref{direct summand b function} with Lemma \ref{BS direct summand b function}.

\begin{corollary}\label{BS finite group b function}
Let $G$ be a finite group acting on a smooth affine variety $X$ of pure dimension, and consider the morphism $\pi: X \to Y$, where $Y=X/G$. Given regular functions $f_1,\dots,f_r\in \os_Y(Y)$ and a section $v \in \Gamma(Y,B_{\bff}^{\omega})$, where $\bff$ is the $r$-tuple $(f_1,\dots,f_r)$, we have $\BS_v\supset \BS_{\pi^*v}$.
\end{corollary}

\subsection{Coherent modules and good filtrations}\label{BS coherent modules and good filtrations}

Throughout this section, we fix an integer $r \geq 1$ and a variety $X$ such that the morphism $\rSpec_X(\gr^F_{\bullet}\ds_X) \to X$ is of finite type. We denote the standard coordinates on $\bba^r$ by $t_1,\dots,t_r$ and put $\rs_X = \ds_X \otimes_{\bbc} D_{\bba^r}$.

\begin{lemma}\label{BS finite type assumption}
The rings $\Gamma(V, W^0\ds_{X\times \bba^r})$, $\Gamma(U, W^0\rs_X)$, and $\Gamma(U, R^+_W\rs_X)$ are noetherian for affine open subsets $V \subset X \times \bba^r$ and $U\subset X$.
\end{lemma}
\proof The argument is essentially the same as that of Lemma \ref{finite type assumption}, so we omit it.\qed

Therefore, Lemma \ref{coherent char} provides a characterization of coherent $V^0\ds_{X\times \bba^r}$-, $V^0\rs_X$-, and $R^+_V\ds_{X\times \bba^r}$-modules. The proofs of the following two results are essentially the same as those of Lemma \ref{surjection} and Corollary \ref{torsion}, respectively, so we omit them.

\begin{lemma}\label{BS surjection}
Given a coherent right $W^0\ds_{X\times \bba^r}$-module $\ms$, there exists an $\os_{X \times \bba^r}$-coherent subsheaf $\fs \subset \ms$ that generates $\ms$ as a $W^0\ds_{X\times \bba^r}$-module.
\end{lemma}

\begin{corollary}\label{BS torsion}
Let $H=V(t_{1}\cdots t_{r}) \subset \bba^r$. Given a coherent right $W^0\ds_{X\times \bba^r}$-module $\ms$ supported on the closed subset $X\times H \subset X\times \bba^r$, there exists an integer $N \geq 0$ such that $\ms \cdot W^N\ds_{X\times \bba^r}=0$.
\end{corollary}

Next, we consider the notion of good filtrations for $W^{0}\rs_{X}$-modules.

\begin{definition}
Let $\ns$ be a right $W^0\rs_X$-module. A good filtration $U^{\bullet}\ns=(U^k\ns)_k$ of $\ns$ is a $\bbz_{\geq 0}$-indexed decreasing filtration of $\ns$ by $W^0\rs_X$-submodules $U^k\ns \subset \ns$ such that the following properties hold:
\begin{enumerate}
\item[i)] $U^0\ns = \ns$.
\item[ii)] $U^i\ns \cdot W^j\rs_X \subset U^{i+j}\ns$ for $i,j\geq 0$.
\item[iii)] $\bigoplus_{k\geq 0} U^k\ns T^k$ is a coherent $R^+_W\rs_X$-module.
\end{enumerate}
The corresponding definition for left $W^0\rs_X$-modules is obtained in the obvious way.
\end{definition}

The proof of the following lemma is standard.

\begin{lemma}\label{BS artin rees}
Given a good filtration $U^{\bullet}\ns$ of a coherent right $W^0\rs_X$-module $\ns$, there exists a constant $c \in \bbz_{\geq 0}$ such that $U^{k + c}\ns \subset \ns \cdot W^{k}\rs_X$ for every $k\geq 0$. The corresponding statement for left $W^0\rs_X$-modules also holds.
\end{lemma}

\subsection{Transfer bimodule and direct image}\label{}

Throughout this section, we fix an integer $r\geq 1$ and a morphism $\pi: Y \to X$ of varieties such that $X=\Spec A$ is a smooth affine variety. We set $\rs_Y = \ds_Y \otimes_{\bbc} D_{\bba^r}$ and $\rs_X = \ds_X \otimes_{\bbc} D_{\bba^r}$. Recall from Section \ref{transfer bimodule and direct image} that $\ds_{Y\to X}=\pi^{*}\ds_{X}$ has a natural $(\ds_Y, \pi^{-1}\ds_X)$-bimodule structure. Also, recall that $\widetilde{\pi}_+(\ns) = \hs^0 R\pi_*(\ns \otimes^L_{\ds_Y} \ds_{Y \to X})$ for a right $\ds_Y$-module $\ns$.

It is immediate that
$$
\pi^*W^0\rs_X = \ds_{Y\to X} \otimes_{\bbc} W^0D_{\bba^r}
$$
is a $(W^0\rs_Y, \pi^{-1}W^0\rs_X)$-bimodule, and that
$$
\pi^*R^+_W\rs_X= \ds_{Y\to X}  \otimes_{\bbc} R^+_WD_{\bba^r}
$$
is an $(R^+_W\rs_Y, \pi^{-1}R^+_W\rs_X)$-bimodule. If $\ns$ is a right $W^0\rs_Y$-module, then $\widetilde{\pi}_+(\ns)$ is a right $W^0\rs_X$-module since
$$
\widetilde{\pi}_+(\ns) = \hs^0 R\pi_*(\ns \otimes^L_{W^0\rs_Y} \pi^*W^0\rs_X).
$$
Similarly, if $\ns= \bigoplus_{k\in \bbz} \ns_k T^k$ is a graded right $R^+_W\rs_Y$-module, then $\widetilde{\pi}_+(\ns)$ is a graded right $R^+_W\rs_X$-module since 
$$
\widetilde{\pi}_+(\ns) = \hs^0 R\pi_*(\ns \otimes^L_{R^+_W\rs_Y} \pi^*R^+_W\rs_X).
$$

Consider the morphism $\pi\times 1:Y \times \bba^r \to X \times \bba^r$. Observe that
$$
(\pi \times 1)^*W^0\ds_{X\times \bba^r} = \os_{Y \times \bba^r} \otimes_{(\pi\times 1)^{-1}\os_{X \times \bba^r}} (\pi\times 1)^{-1}W^0\ds_{X\times \bba^r}
$$
is a $(W^0\ds_{Y \times \bba^r}, (\pi\times 1)^{-1}W^0\ds_{X\times \bba^r})$-bimodule.

\begin{definition}
Given a right $W^0\ds_{Y \times \bba^r}$-module $\ns$, we define
$$
\widetilde{(\pi\times 1)}_{W,+}\ns := \hs^{0}R(\pi\times 1)_* (\ns \otimes_{W^0\ds_{Y\times \bba^r}}^L (\pi\times 1)^*W^0\ds_{X\times \bba^r} ),
$$
which is a right $W^0\ds_{X \times \bba^r}$-module.
\end{definition}

The proofs of the following two results are essentially the same as those of Lemmas \ref{box pull} and \ref{push example}, respectively, so we omit them.

\begin{lemma}\label{BS box pull}
We have $L\pi^*W^0\rs_X= \pi^*W^0\rs_X$, $L\pi^*R^+_W\rs_X= \pi^*R^+_W\rs_X$, and 
$$
L(\pi \times 1)^*W^0\ds_{X \times \bba^r}= (\pi \times 1)^*W^0\ds_{X \times \bba^r}.
$$
\end{lemma}

\begin{lemma}\label{BS push example}
We have $\widetilde{\pi}_+(W^0\rs_Y) = \pi_*\os_Y \otimes_{\os_X} W^0\rs_X$.
\end{lemma}

\subsection{Projective direct image}\label{BS projective direct image}

Throughout this section, we fix an integer $r\geq 1$ and a projective morphism $\pi: Y \to X$ such that $X$ is a smooth affine variety and the morphism $\rSpec_Y (\gr_{\bullet}^F\ds_Y) \to Y$ is of finite type. We denote the projections $Y \times \bba^r \to Y$ and $X \times \bba^r \to X$ by $p_Y$ and $p_X$, respectively.

\begin{lemma}\label{BS projective coherence}
Let $\ms$ be a coherent right $W^0\ds_{Y\times \bba^r}$-module. Then
$$
\hs^i R(\pi \times 1)_*(\ms \otimes^L_{W^0\ds_{Y\times \bba^r}} (\pi\times 1)^*W^0\ds_{X\times \bba^r})
$$
is a coherent right $W^0\ds_{X\times \bba^r}$-module for all $i\in \bbz$. In particular, $\widetilde{(\pi \times 1)}_{W,+}\ms$ is a coherent right $W^0\ds_{X\times \bba^r}$-module.
\end{lemma}
\proof The argument is essentially the same as that of Lemma \ref{projective coherence} after replacing Lemmas \ref{surjection} and \ref{box pull} with Lemmas \ref{BS surjection} and \ref{BS box pull}, respectively. Thus we omit it.\qed

The proofs of the following three results are essentially the same as those of Lemma \ref{push relation}, Lemma \ref{projective graded coherence}, and Corollary \ref{projective good}, respectively, so we omit them.

\begin{lemma}\label{BS push relation}
Given a coherent right $W^0\ds_{Y\times\bba^{r}}$-module $\ms$, we have the following natural isomorphism of right $W^0\rs_X$-modules:
$$
p_{X,*} (\widetilde{(\pi \times 1)}_{W,+}\ms ) \simeq \widetilde{\pi}(p_{Y,*}\ms).
$$
\end{lemma}

\begin{lemma}\label{BS projective graded coherence}
If $\ns$ is a coherent graded right $R^+_W\rs_Y$-module, then $\widetilde{\pi}_+(\ns)$ is a coherent graded right $R^+_W\rs_X$-module.
\end{lemma}

\begin{corollary}\label{BS projective good}
Let $U^{\bullet}\ns$ be a good filtration of a right $W^0\rs_Y$-module $\ns$. If we put 
$$
\im(\widetilde{\pi}_+(U^k\ns)) := \im(\widetilde{\pi}_+(U^k\ns) \to  \widetilde{\pi}_+(\ns) ),
$$
then $(\im(\widetilde{\pi}_+(U^k\ns)) )_{k\geq 0}$ is a good filtration of $\widetilde{\pi}_+(\ns)$.
\end{corollary}

\subsection{Proof of the theorem}\label{}

This section is devoted to the proof of Theorem \ref{main result for bernstein-sato ideals}. We begin with the following two lemmas.

\begin{lemma}\label{BS hypersurface}
Let $X$ be a smooth variety. Given regular functions $g,f_{1},\dots,f_{r}\in \os_{X}(X)$ and invertible functions $u_{1},\dots,u_{r}\in \os_{X}(X)$, if we denote the $r$-tuples $(f_{1},\dots,f_{r})$ and $(u_{1}f_{1},\dots,u_{r}f_{r})$ by $\bff$ and $\bfh$, respectively, then we have $\BS_{g\delta_{\bff}}=\BS_{g\delta_{\bfh}}$.
\end{lemma}
\proof The proof follows the same idea as in \cite[Lemma 2.6]{DM22}. Let $\rs_{X}=\ds_{X}\otimes_{\bbc}D_{\bba^{r}}$. Consider the isomorphism $\phi\colon\rs_{X}\xrightarrow{\sim} \rs_{X}$ of sheaves of $\bbc$-algebras given by $\phi(p)=p$ for $p\in \os_{X}$, $\phi(\theta)=\theta+\sum_{i=1}^{r}s_{i}\theta(u_{i})u^{-1}_{i}$ for $\theta\in \Der_{\bbc}(\os_{X})$, and $\phi(t_{i})=u_{i}t_{i}$ and $\phi(\partial_{t_{i}})=u^{-1}_{i}\partial_{t_{i}}$ for $1\leq i\leq r$. If we view $B_{\bff}$ as a $\rs_{X}$-module via restriction of scalars along $\phi$, then the $\os_{X}$-module isomorphism $\tau\colon  B_{\bfh}\xrightarrow{\sim}B_{\bff}$ given by $\tau(\partial_t^{\beta}\delta_{\bfh})= (\prod_{i=1}^{r}u^{-\beta_{i}}_{i})\partial_t^{\beta}\delta_{\bff}$ is $\rs_{X}$-linear. Note that $\phi(s_{i})=s_{i}$ for each $i$. By \eqref{BS push generation from zero}, it is easy to see that $\phi(W^{1}\rs_{X})=W^{1}\rs_{X}$. Thus for $b\in \bbc[s_{1},\dots,s_{r}]$, we have $b\cdot g\delta_{\bfh}\in W^{1}\rs_{X}\cdot g\delta_{\bfh}$ if and only if $\tau(b\cdot g\delta_{\bfh})\in \tau(W^{1}\rs_{X}\cdot g\delta_{\bfh})$, which is equivalent to $b\cdot g\delta_{\bff}\in W^{1}\rs_{X}\cdot g\delta_{\bff}$, as desired.\qed

\begin{lemma}\label{BS coordinates}
Let $X$ be a smooth variety with global coordinates $x_{1},\dots,x_{n}\in \os_{X}(X)$. Given nonzero regular functions $g,f_{1},\dots,f_{r}\in \os_{X}(X)$, suppose there exist regular functions $h,u_{1},\dots,u_{r}\in \os_{X}(X)$ with $u_{1},\dots,u_{r}$ invertible such that $g=hx_{1}^{b_{1}}\cdots x_{n}^{b_{n}}$ and, for all $i$, $f_{i}=u_{i}x_{1}^{a_{i1}}\cdots x_{n}^{a_{in}}$. If we denote the $r$-tuple $(f_{1},\dots,f_{r})$ by $\bff$, then we have
$$
\prod_{j=1}^{n}\prod_{\ell=0}^{N}\prod_{m=1}^{\sum_{i=1}^{r}a_{ij}}(a_{1j}s_{1}+\cdots +a_{rj}s_{r}+b_{j}+m+\ell)\in \BS_{g\delta_{\bff}}
$$
for some integer $N\geq 0$. Moreover, if $h$ is invertible, then we may take $N=0$. 
\end{lemma}
\proof By Lemma \ref{BS hypersurface}, we may assume that $u_{i}=1$ for each $i$. Consider $\ns=W^{0}\rs_{X}\cdot g\delta_{\bff}$ and $U^{k}\ns=\ns\cap (W^{k}\rs_{X}\cdot x_{1}^{b_{1}}\cdots x_{n}^{b_{n}}\delta_{\bff})$ for $k\geq 0$. Since $\bigoplus_{k\geq 0}(W^{k}\rs_{X}\cdot x_{1}^{b_{1}}\cdots x_{n}^{b_{n}}\delta_{\bff})T^k$ is a coherent $R_{W}^{+}\rs_{X}$-module, the inclusion
$$
\bigoplus_{k\geq 0} U^k\ns T^k\subset \bigoplus_{k\geq 0}(W^{k}\rs_{X}\cdot x_{1}^{b_{1}}\cdots x_{n}^{b_{n}}\delta_{\bff})T^k
$$
implies that $\bigoplus_{k\geq 0} U^k\ns T^k$ is a coherent $R_{W}^{+}\rs_{X}$-module, hence $(U^{k}\ns)_{k\geq 0}$ is a good filtration of $\ns$. Thus by Lemma \ref{BS artin rees}, there exists an integer $N\geq 0$ such that
$$
(W^{0}\rs_{X}\cdot g\delta_{\bff})\cap (W^{N+1}\rs_{X}\cdot x_{1}^{b_{1}}\cdots x_{n}^{b_{n}}\delta_{\bff})\subset W^{1}\rs_{X}\cdot g\delta_{\bff}.
$$
Of course, if $h$ is invertible, then we may take $N=0$. By Lemma \ref{BS left moving}, it follows that for every element $b(s_{1},\dots,s_{r})\in \BS_{x_{1}^{b_{1}}\cdots x_{n}^{b_{n}}\delta_{\bff}}$, we have $\prod_{\ell=0}^{N}b(s_{1}+\ell,\dots,s_{r}+\ell)\in \BS_{g\delta_{\bff}}$. Therefore, it suffices to show that
$$
\prod_{j=1}^{n}\prod_{m=1}^{\sum_{i=1}^{r}a_{ij}}(a_{1j}s_{1}+\cdots +a_{rj}s_{r}+b_{j}+m)\in \BS_{x_{1}^{b_{1}} \cdots x_{n}^{b_{n}}\delta_{\bff}},
$$
which follows from
$$
\prod_{j=1}^{n}\partial_{x_{j}}^{\sum_{i=1}^{r}a_{ij}}  \cdot \prod_{j=1}^{n}x_{j}^{b_{j}}\prod_{i=1}f_{i}\delta_{\bff}=\prod_{j=1}^{n}\prod_{m=1}^{\sum_{i=1}^{r}a_{ij}}(a_{1j}s_{1}+\cdots +a_{rj}s_{r}+b_{j}+m)\prod_{j=1}^{n}x_{j}^{b_{j}}\delta_{\bff}.\qed
$$

\noindent{\it Proof of Theorem \ref{main result for bernstein-sato ideals}}. Recall that $p$ and $q$ denote the morphisms $U/G \onto Y$ and $U\to U/G$, respectively, and that $D\subset X$ is the divisor defined by $\prod_{i =1}^{r}f_i$. Let $p_Y$ and $p_X$ denote the projections $Y \times \bba^r \to Y$ and $X\times \bba^r \to X$, respectively, and put $\rs_Y=p_{Y,*}\ds_{Y\times \bba^r}$ and $\rs_X=p_{X,*}\ds_{X\times \bba^r}$. By Lemma \ref{quot sing finite type}, we can use the results in Sections \ref{BS coherent modules and good filtrations} and \ref{BS projective direct image}.

Let $P\in U$ be a point. By the assumptions of the theorem, there exists an open neighborhood $U_{P}$ of $P$ with global coordinates $y_{1},\dots,y_{n}\in \os_{U}(U_{P})$ such that if we denote by $\pi_P$ the morphism $U_P \to X$ and by $g_P=\det ( (\frac{\partial (\pi_P^* x_k)}{\partial y_l})_{kl} )\in \os_U(U_P)$ the Jacobian of $\pi_P$, then $\pi^{*}_Pf_{i}=u_{i}y_{1}^{a_{i1}}\cdots y_{n}^{a_{in}}$ for each $i$ and $g_P=uy_{1}^{b_{1}}\cdots y_{n}^{b_{n}}$ for some invertible functions $u,u_{1},\dots,u_{r}\in \os_{U}(U_{P})$. We denote the $r$-tuple $(\pi^{*}_Pf_{1},\dots,\pi^{*}_Pf_{r})$ by $\pi_P^*\bff$. By Lemma \ref{BS coordinates}, there exists an integer $N_{P}\geq 0$, with $N_{P}= 0$ if $g=1$, such that $\BS_{\pi^*(g) g_P\delta_{\pi^{*}_{P}\bff}}$ contains the polynomial
$$
\prod_{E}\prod_{\ell=0}^{N_{P}}\prod_{m=1}^{\sum_{i=1}^{r}\ord_{E}(f_{i})}(\ord_{E}(f_{1})s_{1}+\cdots + \ord_{E}(f_{r})s_{r}+ \ord_{E}(g)+k_{E}+m+\ell),
$$
where the outermost product runs over the irreducible components $E$ of $\pi_{P}^{*}D$. 

Consider the open cover $U=\bigcup_{P\in U}U_{P}$, and let $P_{1},\dots,P_{l}\in U$ be finitely many points such that $U=\bigcup_{i=1}^{l}U_{P_{i}}$. If we put $N':=\max_{1\leq i\leq l}N_{P_{i}}$ and
$$
b(s_{1},\dots,s_{r}) :=\prod_{E}\prod_{\ell'=0}^{N'}\prod_{m=1}^{\sum_{i=1}^{r}\ord_{E}(f_{i})}(\ord_{E}(f_{1})s_{1}+\cdots + \ord_{E}(f_{r})s_{r}+ \ord_{E}(g)+k_{E}+m+\ell'),
$$
where the outermost product runs over the irreducible components $E$ of $(\pi \circ p \circ q)^*D$, then we have
\begin{equation}\label{BS first containment}
b(s_{1},\dots,s_{r})\in 
\bigcap_{i=1}^{l}\BS_{\pi^*(g) g_P\delta_{\pi^{*}_{P_{i}}\bff}}.
\end{equation}

Now consider the section
$$
v := (gdx_1\wedge \cdots \wedge dx_n) \delta^{\omega}_{\bff} \in \Gamma(X,B_{\bff}^{\omega}),
$$
and set $v^* := \pi^*v \in \Gamma(Y,B_{\pi^*\bff}^{\omega})$ and $b^{*}(s_{1},\dots,s_{r}):=b(-s_{1}-1,\dots,-s_{r}-1)$. By Remark \ref{relative coincidence}, we have $\BS_{\bff,g}=\BS_{g\delta_{\bff}}$. Therefore, by Lemma \ref{BS left right switch}, it suffices to show that
$$
\prod_{\ell=0}^{N}b^{*}(s_{1}-\ell,\dots,s_{r}-\ell)\in \BS_{v}
$$
for some $N \geq 0$, which the rest of this proof is devoted to showing.

We begin by showing that
\begin{equation}\label{BS second containment}
b^{*}(s_{1},\dots,s_{r})\in 
\BS_{v^*}.
\end{equation} 
 Since
$$
((\pi \circ p \circ q)^*v )|_{U_P} = (\pi^*(g) g_Pdy_1\wedge \cdots \wedge dy_n) \delta^{\omega}_{\pi_P^*\bff} \in \Gamma(U_P,B_{\pi_P^*\bff}^{\omega})
$$
for each $P\in U$, by Lemma \ref{BS left right switch} and \eqref{BS first containment}, we have
$$
b^{*}(s_{1},\dots,s_{r})\in 
\bigcap_{i=1}^{l}\BS_{((\pi \circ p \circ q)^*v )|_{U_{P_{i}}}}=\BS_{(\pi \circ p \circ q)^*v},
$$
where the equality follows from Remark \ref{BS b function local}. By Corollary \ref{BS finite group b function} and Theorem \ref{BS etale pull b}, we have $\BS_{q^*(p^*v^*)}\subset \BS_{p^*v^*}=\BS_{v^*}$, and by \eqref{composition}, we have $q^*(p^*v^*) = (\pi \circ p \circ q)^*v$. Thus \eqref{BS second containment} follows since $b^{*}(s_{1},\dots,s_{r})\in\BS_{(\pi \circ p \circ q)^*v}\subset \BS_{v^*}$.

Next, consider the coherent right $W^0\rs_Y$-module $\ns := v^* \cdot W^0\rs_Y \subset B_{\pi^*\bff}^{\omega}$ and the morphism $\phi: W^0\rs_Y \onto \ns,P \mapsto v^{*}\cdot P$. Note that $\widetilde{\pi}_+(W^0\rs_Y)= W^0\rs_X$ by Lemma \ref{BS push example} and Zariski's Main Theorem. We denote by 
$$
\widetilde{v}:=\widetilde{\pi}_+(\phi)(1)  \in \Gamma(X,\widetilde{\pi}_+(\ns) )
$$
the image of $1\in W^0\rs_X$ under the morphism $\widetilde{\pi}_+(\phi)\colon W^0\rs_X=\widetilde{\pi}_+(W^0\rs_Y) \to \widetilde{\pi}_+(\ns)$.

We claim that there exists some $M \geq 0$ such that
\begin{equation}\label{BS squish}
\widetilde{\pi}_+(\ns) \cdot W^M\rs_X \subset \widetilde{v} \cdot W^0\rs_X.
\end{equation}
To see this, let $v^{**} \in \Gamma(X\times \bba^r,\hs^r_{\Gamma_{\pi^*\bff}}(\omega_{Y\times \bba^r}))$ denote the section corresponding to $v^*$ via
$$
\Gamma(X\times \bba^r,\hs^r_{\Gamma_{\pi^*\bff}}(\omega_{Y\times \bba^r})) = \Gamma(X, B_{\pi^*\bff}^{\omega}).
$$
We also consider the surjection $\phi': W^0\ds_{Y\times \bba^r} \onto v^{**}\cdot W^0\ds_{Y\times \bba^r},P \mapsto v^{**}\cdot P$ and set
$$
\qs = \coker (\widetilde{(\pi\times 1)}_{W,+}(\phi'):  \widetilde{(\pi\times 1)}_{W,+}(W^0\ds_{Y\times \bba^r})
\to
\widetilde{(\pi\times 1)}_{W,+}(v^{**}\cdot W^0\ds_{Y\times \bba^r})).
$$
Note that $\widetilde{(\pi\times 1)}_{W,+}(v^{**} \cdot W^0\ds_{Y \times \bba^r})$ is supported on $\Gamma_{\bff}\subset X \times \bba^r$ since $v^{**} \cdot W^0\ds_{Y \times \bba^r}$ is supported on $\Gamma_{\pi^*\bff}=(\pi\times 1)^{-1}(\Gamma_{\bff}) \subset Y\times \bba^r$. Since $\pi\times 1$ is an isomorphism over $(X \setminus D)\times \bba^r$, it follows that $\widetilde{(\pi\times 1)}_{W,+}(\phi')|_{(X \setminus D) \times \bba^r}$ is surjective, hence $\qs$ is supported on $D\times \bba^r \subset X\times \bba^r$. Thus $\qs$ is supported on $\Gamma_{\bff} \cap (D\times \bba^r) \subset X\times H$, where $H=V(t_{1}\cdots t_{r})\subset \bba^{r}$. By Lemmas \ref{BS projective coherence} and \ref{BS torsion}, there exists some $M \geq 0$ such that $\qs \cdot W^M\ds_{X\times \bba^r} = 0$, which implies that $p_{X,*}(\qs) \cdot W^M\rs_X =0$. Since $p_{X,*}(\qs) = \coker (\widetilde{\pi}_+(\phi))= \widetilde{\pi}_+(\ns)/ \widetilde{v}\cdot W^0\rs_X$ by Lemma \ref{BS push relation}, the claim then follows.

Consider the good filtration $(v^* \cdot W^k\rs_Y)_{k\geq 0}$ of $\ns$. By Lemma \ref{BS projective good}, if we put
$$
\im(\widetilde{\pi}_+(v^* \cdot W^k\rs_Y)) := \im( \widetilde{\pi}_+(v^*  \cdot W^k\rs_Y) \to \widetilde{\pi}_+(\ns)  ),
$$
then $(\im(\widetilde{\pi}_+(\ns \cdot W^k\rs_Y)))_{k\geq 0}$ is a good filtration of $\widetilde{\pi}_+(\ns)$. By Lemma \ref{BS artin rees}, there exists some $N \geq 0$ such that 
\begin{equation}\label{BS artin rees consequence}
\im(\widetilde{\pi}_+(v^* \cdot W^{N+1}\rs_Y)) \subset \widetilde{\pi}_+(\ns) \cdot W^{M+1}\rs_X.
\end{equation}
Now set
$$
a(s_{1},\dots,s_{r}):=\prod_{\ell=0}^{N}b^{*}(s_{1}-\ell,\dots,s_{r}-\ell).
$$
By \eqref{BS second containment} and Lemma \ref{BS right moving}, the morphism $a(s_{1},\dots,s_{r})\colon  \ns \to \ns, w \mapsto w \cdot a(s_{1},\dots,s_{r})$, which is $\ds_Y$-linear, factors through $v^*\cdot W^{N+1}\rs_Y \subset \ns$. Thus we can consider the following commutative diagram of right $\ds_Y$-modules:
\[
\begin{tikzcd}
\ns \arrow[r] \arrow[rr, bend right,"{a(s_{1},\dots,s_{r})}"] & v^*\cdot W^{N+1}\rs_Y \arrow[r, hook] & \ns.
\end{tikzcd}
\]
Applying $\widetilde{\pi}_+$ yields the commutative diagram
\[
\begin{tikzcd}
\widetilde{\pi}_+(\ns) \arrow[r] \arrow[rr, bend right,"{a(s_{1},\dots,s_{r})}"] & \widetilde{\pi}_+(v^*\cdot W^{N+1}\rs_Y) \arrow[r] & \widetilde{\pi}_+(\ns),
\end{tikzcd}
\]
whose commutativity implies that $\widetilde{\pi}_+(\ns) \cdot a(s_{1},\dots,s_{r}) \subset \im(\widetilde{\pi}_+(v^*\cdot W^{N+1}\rs_Y))$. Thus we have
$$
\widetilde{\pi}_+(\ns) \cdot a(s_{1},\dots,s_{r})
 \subset \im(\widetilde{\pi}_+(v^* \cdot W^{N+1}\rs_Y)) 
\subset \widetilde{\pi}_+(\ns) \cdot W^{M+1}\rs_X
\subset \widetilde{v} \cdot W^1\rs_X,
$$
where the second and third inclusions follow from \eqref{BS artin rees consequence} and \eqref{BS squish}, respectively. Since $\widetilde{v}$ is a section of $\widetilde{\pi}_+(\ns)$, we then obtain
\begin{equation}\label{BS pushed b functional equation}
\widetilde{v}  \cdot a(s_{1},\dots,s_{r}) \in\widetilde{v} \cdot W^1\rs_X.
\end{equation}

We next show that there is a morphism of right $W^0\rs_X$-modules $\psi: \widetilde{v} \cdot W^0\rs_X  \to v \cdot W^0\rs_X$ such that $\psi(\widetilde{v})=v$. Let $j: X \setminus D  \to X$ denote the inclusion. Note that the natural morphism $v\cdot W^0\rs_X \to j_*j^*(v\cdot W^0\rs_X)$ is injective since $v\cdot W^0\rs_X \subset B_{\bff}^{\omega}$ and $B_{\bff}^{\omega}$ is $\os_X$-torsion-free by Lemma \ref{bfomega action}. Thus it suffices to show that there is a morphism of right $W^0\rs_X$-modules $\widetilde{v} \cdot W^0\rs_X \to j_*j^*(v\cdot W^0\rs_X)$ that maps $\widetilde{v}$ to $v|_{X \setminus D}$. Since $\pi$ is an isomorphism over $X \setminus D$, it is straightfoward to see that $j^*(\widetilde{v} \cdot W^0\rs_X)=j^*(v\cdot W^0\rs_X)$ and $\widetilde{v}|_{X \setminus D} = v|_{X \setminus D}$. The desired morphism is then the composition $\widetilde{v} \cdot W^0\rs_X \to j_*j^*(\widetilde{v} \cdot W^0\rs_X) =  j_*j^*(v\cdot W^0\rs_X)$. 

Finally, since $\widetilde{v}\cdot a(s_{1},\dots,s_{r}) = \widetilde{v} \cdot P$ for some $P\in W^1\rs_X$ by \eqref{BS pushed b functional equation}, it follows that 
$$
v\cdot a(s_{1},\dots,s_{r}) = \psi(\widetilde{v}\cdot a(s_{1},\dots,s_{r}) ) = \psi(\widetilde{v} \cdot P)= v \cdot P
$$
by $W^0\rs_X$-linearity of $\psi$. Hence $a(s_{1},\dots,s_{r})\in \BS_{v}$, as desired. \qed

\appendix

\section{Technical lemmas}\label{technical lemmas}

In this section, we collect several technical lemmas used throughout the paper.

\begin{remark}\label{push bounded coho dim}
Given a morphism $f: X \to Y$ of varieties, note that $R^if_*(\fs)=0$ for all sheaves $\fs$ of abelian groups on $X$ and $i>\dim X$ (see \cite[Theorem III.2.7 and Proposition III.8.1]{Har77}). Thus $Rf_*$ is defined on chain complexes that are not necessarily bounded below by \cite[Lemma 13.32.2]{Stacks}.
\end{remark}

\begin{lemma}\label{bounded cohomological dimension}
Let $f: X \to Y$ be a morphism of varieties, and set $d=\dim X$. Given a chain complex $\cs^{\bullet}$ of sheaves of abelian groups on $X$, we have
$$
\hs^i Rf_*(\cs^{\bullet}) = \hs^iRf_* (\sigma_{\geq i-d-1}\cs^{\bullet})
$$
for all $i\in \bbz$, where $\sigma$ denotes the stupid truncation.
\end{lemma}
\proof Let $\is=\im(\cs^{i-d-2} \to \cs^{i-d-1})$ and consider the short exact sequence of chain complexes
$$
0 \to \sigma_{\geq i-d-1}\cs^{\bullet} \to \tau_{\geq i-d-1}\cs^{\bullet} \to \is[-i+d+2]  \to 0,
$$
where $\tau$ denotes the canonical truncation. By Remark \ref{push bounded coho dim} and \cite[Lemma 13.32.2]{Stacks} we have $\hs^i Rf_*(\cs^{\bullet})= \hs^i Rf_*(\tau_{\geq i-d-1}\cs^{\bullet})$. Thus, applying $Rf_*$ and taking cohomology yields the exact sequence
$$
R^{d+1}f_*(\is) \to \hs^i Rf_*(\sigma_{\geq i-d-1}\cs^{\bullet}) \to \hs^i Rf_*(\cs^{\bullet}) \to R^{d+2}f_*(\is).
$$
Since the outer terms vanish, the claim follows. \qed

\begin{remark}\label{filtered colimit direct sum}
Given a family $\{\fs_i\}_{i\in I}$ of sheaves of abelian groups on a topological space $X$, we may view the direct sum $\bigoplus_{i\in I}\fs_i$ as a filtered colimit, running over finite subsets $J$ of $I$, of the partial direct sums $\bigoplus_{j\in J}\fs_j$.
\end{remark}

\begin{lemma}\label{qcoh direct summands}
Let $\fs$ be a quasicoherent sheaf on a scheme $X$. If $\{\fs_i\}_{i\in I}$ is a family of $\os_X$-module subsheaves of $\fs$ such that $\fs=\bigoplus_{i\in I} \fs_i$, then $\fs_i$ is quasicoherent for all $i \in I$. 
\end{lemma}
\proof Note that $\fs(U)=\bigoplus_{i\in I} \fs_i(U)$ for all affine opens $U \subset X$ by Remark \ref{filtered colimit direct sum} and \cite[Lemma 6.29.1]{Stacks}. Given an affine open $U=\Spec R \subset X$ and $f\in R$, it suffices to show that the natural map $\phi_i: \fs_i(U)\otimes_R R_f \to \fs_i(\Spec R_f)$ is an isomorphism for each $i$. This follows since $\bigoplus_{i\in I}\phi_i$ is equal to the isomorphism $\fs(U)\otimes_R R_f \xrightarrow{\sim} \fs(\Spec R_f)$. \qed

\begin{remark}\label{flabby direct sum}
Let $X$ be a noetherian topological space and let $\{\fs_i\}_{i\in I}$ be an arbitrary family of sheaves of abelian groups on $X$. By Remark \ref{filtered colimit direct sum} and \cite[Lemma 6.29.1]{Stacks}, we have $\Gamma(U, \bigoplus_{i\in I} \fs_i) = \bigoplus_{i\in I} \Gamma(U,\fs_i)$ for every open subset $U \subset X$. We record the following consequences
\begin{enumerate}
\item[i)] If $\fs_i$ is flabby for each $i\in I$, then $\bigoplus_i \fs_i$ is also flabby. 
\item[ii)] If $f: X \to Y$ is any continuous map, then $f_*( \bigoplus_i \fs_i) = \bigoplus_i f_*\fs_i$. 
\end{enumerate}
\end{remark}

\begin{lemma}\label{derived direct image and direct sum}
Let $f: X \to Y$ be a morphism of varieties and let $\rs$ be a sheaf of (possibly noncommutative) rings on $X$. Given a left $\rs$-module $\gs$ and an arbitrary family $\{\fs_i\}_{i\in I}$ of right $\rs$-modules, the natural morphism
$$
\bigoplus_{i\in I}  \hs^j (Rf_*(\fs_i\otimes^L_{\rs} \gs)) \to \hs^j(Rf_*( (\bigoplus_{i\in I} \fs_i ) \otimes^L_{\rs} \gs ) )
$$
is an isomorphism for all $j\in \bbz$.
\end{lemma}
\proof By taking an $\rs$-flat resolution of $\gs$, we see that it suffices to show that if $\{\fs^{\bullet}_i\}_{i\in I}$ is an arbitrary family of chain complexes of sheaves of abelian groups on $X$, then the natural morphism
$$
\bigoplus_{i\in I}  \hs^j (Rf_*(\fs_i^{\bullet})) \to \hs^j (Rf_*( \bigoplus_{i\in I} \fs_i^{\bullet} ) )
$$
is an isomorphism for all $j\in \bbz$. By \cite[Lemma 13.32.1]{Stacks}, there exists a quasi-isomorphism $\fs_i^{\bullet} \to \gs_i^{\bullet}$ such that $\gs_i^{\bullet}$ is a chain complex of flabby sheaves. Consider the quasi-isomorphism $\bigoplus_{i\in I} \fs_i^{\bullet} \to \bigoplus_{i\in I} \gs_i^{\bullet}$. By Remark \ref{flabby direct sum}, $\bigoplus_{i\in I} \gs_i^{\bullet}$ is a chain complex of flabby sheaves. We are done since
$$
\bigoplus_i \hs^j (Rf_*(\fs_i^{\bullet})) = \bigoplus_i \hs^j (f_*\gs_i^{\bullet}) =  \hs^j(\bigoplus_i f_*\gs_i^{\bullet}) =  \hs^jf_*(\bigoplus_i \gs_i^{\bullet}) =\hs^j (Rf_*( \bigoplus_i \fs_i^{\bullet} ) )
$$
where the first and last isomorphism follow from \cite[Lemma 13.32.2]{Stacks} and the third isomorphism follows from Remark \ref{flabby direct sum}. \qed

We leave the proof of the following easy lemma as an exercise to the reader.

\begin{lemma}\label{transfer bimodule tool}
Given a morphism of affine schemes $\pi: Y \to X$ and quasicoherent sheaves of rings $\as$ and $\bs$ on $Y$ and $X$, respectively, assume that there exists a left $\as$-module structure on $\pi^*\bs = \os_Y \otimes_{\pi^{-1}\os_X}\pi^{-1}\bs$ extending its left $\os_Y$-module structure such that $\Gamma(Y,\pi^*\bs)$, equipped with the induced actions of $\as(Y)$ and $\bs(X)$, is an $(\as(Y),\bs(X))$-bimodule. Then $\pi^*\bs$ is an $(\as, \pi^{-1}\bs)$-bimodule.
\end{lemma}

\section{Quotient singularities}\label{quotient singularities}

As in \cite[Definition 4.27]{Yas04}, we use the following definition of quotient singularities.

\begin{definition}\label{quot sing definition}
We say that a variety $Y$ has quotient singularities if there exists an \'etale surjection $\coprod_{i\in I}U_{i}/G_{i} \onto Y$, where $I$ is a finite set and $G_{i}$ is a finite group acting on a smooth affine variety $U_{i}$ for each $i\in I$.
\end{definition}

\begin{lemma}\label{quot sing etale cover}
A variety $Y$ has quotient singularities if and only if there exists an \'{e}tale surjection $U/G \onto Y$, where $G$ is a finite group acting on a smooth affine variety $U$.
\end{lemma}
\proof Indeed, if $\coprod_iU_{i}/G_{i} \onto Y$ is an \'etale surjection as in Definition \ref{quot sing definition}, then by setting $U=\coprod_i U_i$ and $G=\prod_i G_i$, we have $U/G=\coprod_i U_i/G_i$. \qed

The rest of this section is devoted to proving the following two facts.

\begin{theorem}\label{quot sing equivalence}
A variety $Y$ has quotient singularities if and only if for every closed point $y\in Y$, there exists a finite subgroup $G \subset \GL_n(\bbc)$ and an isomorphism $\os_{Y,y}^{\wedge} \cong \bbc[[t_1,\dots,t_n]]^G$ of $\bbc$-algebras.
\end{theorem}

\begin{lemma}\label{quot sing finite type}
Let $Y$ be a variety with quotient singularities. Then the natural morphism $\rSpec_Y( \gr^F_{\bullet}\ds_Y) \to Y$ is of finite type.
\end{lemma}

Before we prove these two facts, we begin with a few auxiliary lemmas and remarks.

\begin{lemma}\label{invariant completion}
Let $G$ be a finite group acting on a local noetherian $\bbc$-algebra $(A,m)$. Then we have a natural isomorphism $(A^G)^{\wedge} \xrightarrow{\sim} (A^{\wedge})^G$.
\end{lemma}
\proof Because $A^G \to A$ is an integral extension, $A^G$ is a local ring with maximal ideal $m^G$ and $\sqrt{A \cdot m^G} = m$. Fix $l\geq 1$ such that $m^l \subset A \cdot m^G$, which implies that $m^{il} \subset A \cdot (m^G)^i$ for every $i\geq 0$. Since inverse limits preserve left exactness, the inverse limit of the exact sequence
$$
0 \to (A/m^i)^G \to A/m^i \to \bigoplus_{g\in G} A/m^i
$$
is also exact, hence $(A^{\wedge})^G \simeq \varprojlim ( (A/m^i)^G)$. Since taking $G$-invariants is exact, we have $A^G/(m^i)^G \xrightarrow{\sim} (A/m^i)^G$, hence $(A^{\wedge})^G \simeq\varprojlim A^G/ (m^i)^G$. Since $A^G \to A$ is a direct summand inclusion of $A$-modules, we have $(A \cdot (m^G)^i) \cap A^G = (m^G)^i$, hence
$$
(m^{il})^G = m^{il} \cap A^G \subset (A \cdot (m^G)^i) \cap A^G = (m^G)^i.
$$
Because $(m^G)^{i} \subset (m^i)^G$ and $(m^{il})^G \subset (m^G)^i$ for every $i\geq 0$, we obtain
$$
(A^{\wedge})^G \simeq \varprojlim A^G/ (m^i)^G \simeq \varprojlim A^G/ (m^G)^i = (A^G)^{\wedge},
$$
as desired. \qed

\begin{lemma}\label{stabilizer}
Given a ring $A_i$ with a prime ideal $p_i \subset A_i$ for $i=1,\dots,r$ such that each $\Spec A_i$ is connected, let $G$ be a finite group acting on the product $A=\prod_{i=1}^r A_i$ by ring automorphisms such that $G$ acts transitively on the subset
$$
\textstyle \{[p_1],\dots,[p_r]\} \subset \coprod_{i=1}^r \Spec A_i = \Spec A.
$$
If we denote the set-theoretic stabilizer of $[p_1]$ by $G_1 = \{g\in G \ | \ g(p_1)=p_1\}$, then the following map induced by the projection is an isomorphism:
$$
A^G \xrightarrow{\sim} A_1^{G_1}.
$$
\end{lemma}
\proof For each $g\in G$ and $i\in \{1,\dots,r\}$, let $g(i) \in \{1,\dots,r\}$ denote the unique index given by $g(p_{i}) = p_{g(i)}$. Since $\Spec A_i$ are the connected components of $\Spec A$ and $G$ acts continuously on $\Spec A$, the action of $g\in G$ on $A$ can be expressed as a direct product of ring isomorphisms $g_i: A_{i} \xrightarrow{\sim} A_{g(i)}$ over $i\in \{1,\dots,r\}$. Given $g\in G$, we write $g=\prod_{i=1}^r g_i$ for such a direct product.

To prove injectivity, let $x=(x_i)_i\in A^G$ be an element such that $x_1=0$. To show that $x_i=0$ for each $i$, note that by transitivity there exists $g\in G$ such that $g(1)=i$, so by $G$-invariance of $x$ we have $x_i = g_1(x_1)=0$. To prove surjectivity, let $x_1 \in A_1^{G_1}$. For each $i\in \{1,\dots,r\}$, we put $x_i = g_1(x_1)$ for any $g\in G$ such that $g(1)=i$, which exists by transitivity. This is well-defined because if $g,h\in G$ are such that $g(1)=i$ and $h(1)=i$, then $(h^{-1}g)(1)=1$ so that $h^{-1}g\in G_1$ and 
$$
g_1(x_1) = h_1 (h_i)^{-1}g_1 (x_1) =h_1(x_1).
$$
It is then immediate that $(x_i)_i \in A^G$, so we are done. \qed

\begin{lemma}[{\cite[Theorem 1]{Boc45}, \cite[Lemma 1]{Car53}}]\label{linearizing power series}
Given a finite group $G$ acting on a power series ring $D=\bbc[[t_1,\dots,t_n]]$ by $\bbc$-algebra automorphisms, there exists a $\bbc$-algebra automorphism $\phi$ of $D$ such that $g(\phi(t_i)) = \sum_j a_{g ij} \phi(t_j)$ for some $a_{g ij} \in \bbc$ for every $g\in G$, that is, the $G$-action is linear with respect to the coordinates $\phi(t_1),\dots,\phi(t_n)$.
\end{lemma}
\proof Let $m=(t_1,\dots,t_n)$ denote the maximal ideal of $D$. Note that a local $\bbc$-algebra map $f: (D,m) \to (D,m)$ is an automorphism iff the induced map $\gr^1_m f: m/m^2 \to m/m^2$ is an isomorphism. For each $g\in G$, let $L_g$ denote the unique $\bbc$-algebra automorphism of $D$ such that $L_g(t_i) \in \sum_{j=1}^n \bbc t_j$ and $g(t_i) = L_g(t_i) \text{ mod } m^2$ (indeed, such a $\bbc$-algebra map is an automorphism because $L_g$ and $g$ induce the same endomorphism of $m/m^2$). By uniqueness, we have $L_{gh}=L_g \circ L_h$ for $g,h\in G$. Consider the $\bbc$-module map
$$
\sigma = \frac{1}{|G|} \sum_{g\in G} g \circ L_g^{-1}: D \to D.
$$
It is immediate that $g \circ \sigma = \sigma \circ L_g$ for every $g\in G$. Because each $g \circ L_g^{-1}$ is a $\bbc$-algebra automorphism of $D$ inducing the identity on $m/m^2$, we have that $\sigma(m^j) \subset m^j$ for all $j\geq 0$ and the induced map $\gr^1_m \sigma : m/m^2 \to m/m^2$ is the identity. Let $\phi: D \to D$ denote the $\bbc$-algebra map given by $\phi(t_i) = \sigma(t_i)$ for each $i$, which is an automorphism because $\gr^1_m \phi = \gr^1_{m}\sigma$. Write $L_g(t_i) = \sum_j a_{g ij} t_j$ with $a_{gij}\in \bbc$. Then $\phi$ is the desired automorphism because
$$
\textstyle g \phi(t_i) = g \sigma(t_i) = \sigma L_g(t_i) = \phi L_g(t_i) = \sum_j a_{g ij} \phi(t_j). \qed
$$

\begin{remark}\label{flat invariants}
Let $G$ be a finite group acting on a ring $S$, and let $R=S^G$. Given a flat $R$-module $M$, note that we have $(M\otimes_R S)^G = M$ since the tensor product of the exact sequence $0\to R \to S \to \bigoplus_{g\in G} S$ with $M$ remains exact.
\end{remark} 

\begin{remark}\label{artin}
Let $y\in Y$ be a closed point on a variety $Y$ such that $\os_{Y,y}^{\wedge} \cong \bbc[[t_1,\dots,t_n]]^G$ for some finite subgroup $G \subset \GL_n(\bbc)$. If we consider the corresponding linear action of $G$ on $\bba^n=\Spec \bbc[t_1,\dots,t_n]$, then by Lemma \ref{invariant completion}, we have 
$$
\os_{\bba^n/G, 0}^{\wedge} \simeq (\os_{\bba^n}^{\wedge})^G = \os_{Y,y}^{\wedge}.
$$
By a standard application of Artin approximation (see \cite[Corollary 1.14]{Alp15}), there exists a common \'{e}tale neighborhood $(W,w)$ of $(Y,y)$ and $(\bba^n/G,0)$. 
\end{remark}

\noindent{\it Proof of Theorem \ref{quot sing equivalence}}. Let $Y$ be a variety. First suppose that $Y$ has quotient singularities. By Lemma \ref{quot sing etale cover}, there exists an \'{e}tale surjection $U/G \onto Y$, where $G$ is a finite group acting on a smooth affine variety $U$. For any closed point $y\in Y$, if $y'\in U/G$ is any point lying above $y$, then $\os_{Y,y} \cong \os_{U/G,y'}$ since $U/G \onto Y$ is \'{e}tale. So we may assume that $Y=U/G$. Fix a closed point $y\in Y$. We denote the morphism $U \to Y=U/G$ by $\pi$ and set $\pi^{-1}(y)=\{x_1,\dots,x_r\}$. By flatness of $\os_Y(Y) \to \os_{Y,y} \to \os_{Y,y}^{\wedge}$ and \cite[Lemma 10.97.8]{Stacks}, we have
$$
\os_{Y,y}^{\wedge} = (\os_U(U) \otimes_{\os_Y(Y)} \os_{Y,y}^{\wedge})^G
= \bigg(\prod_{i=1}^r \os_{U,x_i}^{\wedge}\bigg)^G.
$$
Since local rings have connected spectrum and $G$ acts transitively on the fibers of $\pi$, by Lemma \ref{stabilizer}, we have $\os_{Y,y}^{\wedge} \simeq (\os_{U,x_1}^{\wedge})^{G_1}$, where $G_1 \subset G$ is the set-theoretic stabilizer of $x_1$. By Lemma \ref{linearizing power series}, we can write $\os_{U,x_1}^{\wedge} = \bbc[[t_1,\dots,t_n]]$ such that $G_1$ acts linearly on $t_1,\dots,t_n$. So if we let $G \subset \GL_n(\bbc)$ denote the image of $G_1$, then $\os_{Y,y}^{\wedge} \simeq \bbc[[t_1,\dots,t_n]]^{G}$. 

Conversely, assume that for every closed point $y\in Y$, there exists a finite subgroup $G \subset \GL_n(\bbc)$ and an isomorphism $\os_{Y,y}^{\wedge} \cong \bbc[[t_1,\dots,t_n]]^G$ of $\bbc$-algebras. By Remark \ref{artin}, there exists a finite collection of affine varieties $\{W_i\}_{i\in I}$, an \'{e}tale surjection $\coprod_{i\in I} W_i \onto Y$, and \'{e}tale morphisms $W_i \to \bba^{n_i}/G_i$ for some finite subgroup $G_i \subset \GL_{n_i}(\bbc)$ for each $i\in I$. Note that $U_i=W_i \times_{\bba^{n_i}/G_i}  \bba^{n_i}$ is a smooth affine variety with a natural $G_i$-action pulled back from the one on $ \bba^{n_i}$. Since $W_i \to \bba^{n_i}/G_i$ is flat, we have $U_i/G_i = W_i$ by Remark \ref{flat invariants}, hence $Y$ has quotient singularities. \qed

\noindent{\it Proof of Lemma \ref{quot sing finite type}}. Since the property of a morphism being finite type is \'{e}tale-local (see \cite[Lemma 35.23.12]{Stacks}), by Theorem \ref{quot sing equivalence} and Remark \ref{artin}, it suffices to prove the theorem when $Y=\bba^n/G$ for a finite subgroup $G\subset \GL_n(\bbc)$, but this follows from \cite[III.III Theorem 7]{Kan77}. \qed

\end{document}